\documentclass[10pt]{amsart}

\usepackage{times,amsfonts,amsmath,amstext,amsbsy,amssymb,
  amsopn,amsthm,upref,eucal, amscd}
\usepackage[T1]{fontenc}
\usepackage{color}

\usepackage[pdftex]{graphicx}

\newtheorem{theorem}{Theorem}[section]
\newtheorem{lemma}[theorem]{Lemma}
\newtheorem{corollary}[theorem]{Corollary}
\newtheorem{proposition}[theorem]{Proposition}

\numberwithin{equation}{section}

\theoremstyle{definition}

\newtheorem{remark}[theorem]{Remark}

\def\leq{\leqslant }
\def\geq{\geqslant}

\begin{document}

\title[Fractal geometry of $M\setminus L$]{Fractal geometry of the complement of Lagrange spectrum in Markov spectrum}

\author[C. Matheus and C. G. Moreira]{Carlos Matheus and Carlos Gustavo Moreira}

\address{Carlos Matheus:
Centre de Math\'ematiques Laurent Schwartz, CNRS (UMR 7640), \'Ecole Polytechnique, 91128 Palaiseau, France.
}

\email{matheus.cmss@gmail.com}

\address{Carlos Gustavo Moreira: School of Mathematical Sciences,   
Nankai University, Tianjin 300071, P. R. China, and 
IMPA, Estrada Dona Castorina 110, 22460-320, Rio de Janeiro, Brazil
}

\email{gugu@impa.br}

\date{\today}

\begin{abstract} 
The Lagrange and Markov spectra are classical objects in Number Theory related to certain Diophantine approximation problems. Geometrically, they are the spectra of heights of geodesics in the modular surface. 

These objects were first studied by A. Markov in 1879, but, despite many efforts, the structure of the complement $M\setminus L$ of the Lagrange spectrum $L$ in the Markov spectrum $M$ remained somewhat mysterious. In fact, it was shown by G. Freiman (in 1968 and 1973) and M. Flahive (in 1977) that $M\setminus L$ contains infinite \emph{countable} subsets near 3.11 and 3.29, and T. Cusick conjectured in 1975 that all elements of $M\setminus L$ were $<\sqrt{12}=3.46\dots$, and this was the \emph{status quo} of our knowledge of $M\setminus L$ until 2017.    

In this article, we show the following two results. First, we prove that $M\setminus L$ is \emph{richer} than it was previously thought because it contains a Cantor set of Hausdorff dimension larger than $1/2$ near $3.7$: in particular, this solves (negatively) Cusick's conjecture mentioned above. Secondly, we show that $M\setminus L$ is \emph{not} very thick: its Hausdorff dimension is strictly smaller than one. 
\end{abstract}
\maketitle


\section{Introduction}

The (classical) Lagrange and Markov spectra are subsets of the real line related to Diophantine approximation problems. More precisely, the \emph{Lagrange spectrum} is
$$L:=\left\{\limsup\limits_{\substack{p, q\to\infty \\ p, q\in\mathbb{Z}}} \frac{1}{|q(q\alpha-p)|}<\infty:\alpha\in\mathbb{R}-\mathbb{Q}\right\}$$
and the \emph{Markov spectrum} is
$$M:=\left\{\frac{1}{\inf\limits_{\substack{(x,y)\in\mathbb{Z}^2\\(x,y)\neq(0,0)}} |q(x,y)|}<\infty: q(x,y)=ax^2+bxy+cy^2 \textrm{ real indefinite, }  b^2-4ac=1 \right\}$$

These sets were intensively studied by several authors (including Hurwitz, Frobenius, Perron, ...) since the seminal works \cite{Ma1} and \cite{Ma2} of Markov from 1879 and 1880 establishing (among other things) that 
$$L\cap (-\infty, 3) = M\cap (-\infty, 3) = \left\{\sqrt{5}<\sqrt{8}<\frac{\sqrt{221}}{5}<\dots\right\}$$ consists of an \emph{explicit} increasing sequence of quadratic surds accumulating only at $3$.

Hall \cite{Ha} proved in 1947 that $L\supset [c,\infty)$ for some constant $c>3$. For this reason, a half-line $[c,\infty)$ contained in the Lagrange spectrum is called a \emph{Hall ray}. Freiman \cite{Fr73b} and Schecker \cite{Sch77} proved that $[\sqrt{21},\infty)\subset L$, and Freiman \cite{Fr75} determined in 1975 the biggest half-line $[c_F,\infty)$ contained in the Lagrange spectrum, namely,
$$c_F:=\frac{2221564096+283748\sqrt{462}}{491993569} \simeq 4.5278\dots$$
The constant $c_F$ is called \emph{Freiman's constant}.

In general, it is known that $L\subset M$ are closed subsets of $\mathbb{R}$. The results of Markov, Hall and Freiman mentioned above imply that the Lagrange and Markov spectra coincide below $3$ and above $c_F$. Nevertheless, it took a certain time to decide whether these two sets were the same: in fact, Freiman \cite{Fr68} showed in 1968 that $M\setminus L\neq \emptyset$ by exhibiting a countable subset of isolated points of $M\setminus L$ near $3.11$; after that, Freiman proved in 1973 that $M\L$ contains a point $\alpha_{\infty}$ near $3.29$, and Flahive showed in 1977 that $\alpha_{\infty}$ is the accumulation point of a countable subset of $M\L$ near $3.29$. 

This state of affairs led Cusick \cite{C} to conjecture in 1975 that the Lagrange and Markov spectra coincide after $\sqrt{12}$, i.e., $(M\setminus L)\cap [\sqrt{12},\infty)=\emptyset$: in fact, one reads at page 516 the phrase ``I think it is likely that $L$ and $M$ coincide above $\sqrt{12}=3.46410$''. 

The reader is invited to consult the excellent book \cite{CF} of Cusick--Flahive for a beautiful review of the literature produced on this topic until 1989, and the recent article \cite{Mo} of the second author for more discussions of the fractal geometry of $L$ and $M$. 

\subsection{Statement of the main results}

The first main result of this paper (extending the analysis in our two previous papers \cite{MaMo1}, \cite{MaMo2}) answers Cusick's conjecture by showing that $M\setminus L$ near $3.7$ is \emph{richer} than countable subsets:

\begin{theorem}\label{t.1} The intersection of $M\setminus L$ with the interval $(3.7, 3.71)$ has Hausdorff dimension $>0.53128$ (and, a fortiori, $(M\setminus L)\cap (3.7, 3.71)\neq\emptyset$). 
\end{theorem}  

\begin{remark} We explain in Appendix \ref{s.Cantor-Cusick} that our proof of Theorem \ref{t.1} actually give more details about $M\setminus L$: for instance, our arguments allow to compute the largest \emph{known} element $\Upsilon$ of $M\setminus L$. 
\end{remark}

The second main result in this article says that $M\setminus L$ \emph{doesn't} have full Hausdorff dimension (and, hence, it is \emph{not} very thick): 

\begin{theorem}\label{t.prrova}
$HD(M\setminus L) < 0.986927$.
\end{theorem}

\begin{remark} In Appendix \ref{s.prova}, we give empirical evidence towards the better estimate $HD(M\setminus L)<0.888$. 
\end{remark}

It follows that $M\setminus L$ has empty interior, and so, since $M$ and $L$ are closed subsets of $\mathbb{R}$, $\overline{int(M)}=\overline{int(L)}\subset L\subset M$. In particular, we have the following:

\begin{corollary}
$int(M)=int(L)$.
\end{corollary}

As a consequence, we recover the fact, proved in \cite{Fr75}, that the biggest half-line contained in $M$ coincides with the biggest half-line $[c_F,\infty)$ contained in $L$.

Our main results show that $M\setminus L$ has an intricate structure and this motivates the following question\footnote{Despite the fact that we don't know any point of $L$ accumulated by points in $M\setminus L$.}. Consider the Lagrange spectrum $L$, and denote by $X$ the set obtained from $L$ by removing all non-trivial closed intervals contained on it and all of its isolated points. Is every point of $X$ accumulated by points in $M\setminus L$?

\begin{remark} We expect that the techniques in this article will be helpful in computing exactly the first decimal digit of $HD(M\setminus L)$. 
\end{remark}

Our approach to Theorems \ref{t.1} and \ref{t.prrova} is based on some qualitative \emph{dynamical} insights leading to a series of quantitative estimates with continued fractions. Before explaining this point, let us briefly recall the classical dynamical characterization of the Lagrange and Markov spectra due to Perron (see, e.g., \cite{CF} for more details). 

\subsection{Continued fractions, shift dynamics, and the Lagrange and Markov spectra} 

Given a sequence $a=(a_n)_{n\in\mathbb{Z}}\in(\mathbb{N}^*)^{\mathbb{Z}}$, we denote by $\lambda_i(a):=[a_i;a_{i+1}, a_{i+2},\dots] + [0;a_{i-1}, a_{i-2},\dots]$, where 
$$[c_0;c_1,c_2,\dots]:= c_0+\frac{1}{c_1+\frac{1}{c_2+\frac{1}{\ddots}}}$$
stands for the usual continued fraction development. 

The Markov value $m(a)$ of $a$ is $m(a)=\sup\limits_{n\in\mathbb{Z}} \lambda_n(a)$ and the Lagrange value $\ell(a)$ of $a$ is $\ell(a)=\limsup\limits_{n\to\infty} \lambda_n(a)$. In this setting, the Markov spectrum $M$ is the collection of all finite Markov values and the Lagrange spectrum $L$ is the collection of all finite Lagrange values.

From the point of view of Dynamical Systems, the previous paragraphs can be rewritten as follows. Let $\Sigma=(\mathbb{N}^*)^{\mathbb{Z}} = (\mathbb{N}^*)^{\mathbb{Z}^-} \times (\mathbb{N}^*)^{\mathbb{N}} = \Sigma^-\times \Sigma^+$ and $\pi^{\pm}:\Sigma\to\Sigma^{\pm}$ the natural projections.

Consider $\sigma$ the left-shift dynamics on $\Sigma$, and denote by $f:\Sigma\to\mathbb{R}$ the height function $f((b_n)_{n\in\mathbb{Z}}) := \lambda_0((b_n)_{n\in\mathbb{Z}}) = [b_0; b_1,\dots] + [0; b_{-1},\dots]$. 

The Markov value $m(\underline{b})$ of a sequence $\underline{b}\in\Sigma$ is $m(\underline{b})=\sup\limits_{n\in\mathbb{Z}} f(\sigma^n(\underline{b}))$. Similarly, the Lagrange value $\ell(\underline{b})$ of a sequence $\underline{b}\in\Sigma$ is $\ell(\underline{b})=\limsup\limits_{n\to\infty} f(\sigma^n(\underline{b}))$.

Therefore, we can think geometrically about $L$ and $M$ in terms of the heights of the orbits of a dynamical system $G$ on the plane $\mathbb{R}^2$. Indeed, the natural map $\Sigma^+\times\Sigma^-\to\mathbb{R}\times\mathbb{R}$ sending $(b_i)_{i\in\mathbb{Z}}$ to $([b_0;b_1,\dots], [0;b_{-1},\dots])$ allows to transfer the shift dynamics $\sigma:\Sigma\to\Sigma$ and the height function $f:\Sigma\to\mathbb{R}$ to the plane $\mathbb{R}^2$: in this way, $\sigma$ becomes a natural extension $G:(\mathbb{R}\setminus\mathbb{Q})\times (\mathbb{R}\setminus\mathbb{Q})\cap (0,1))\to (\mathbb{R}\setminus\mathbb{Q})\times (\mathbb{R}\setminus\mathbb{Q})\cap (0,1))$ of the so-called Gauss map and $f$ becomes $\widetilde{f}:\mathbb{R}^2\to\mathbb{R}$, $\widetilde{f}(x,y):=x+y$. As it is explained below, even though our proofs of Theorems \ref{t.1} and \ref{t.prrova} might superficially look a ``lucky'' concatenation of a series of lemmas about continued fractions, they are directly motivated by qualitative dynamical features of the orbits of $G$ with respect to the height function $\widetilde{f}$.    

\subsection{Ideas behind the proof of Theorem \ref{t.1}}

Our first source of inspiration to construct new elements in $M\setminus L$ is provided by Flahive paper \cite{F}. In this article, Flahive introduced the notion of \emph{semi-symmetric} words and she proved that an element of $M\setminus L$ is usually associated to \emph{non} semi-symmetric words. In particular, it is not surprising that Freiman's construction of elements in $M\L$ is related to the non semi-symmetric words (of odd lengths\footnote{We insist on non semi-symmetric words of \emph{odd} length because any modification of the associated infinite periodic sequence will force a definite increasing of the Markov value in one of two consecutive periods.})  $222211221$ and $2112221$, and our construction of new elements in $M\setminus L$ is based on the non semi-symmetric word (of odd length) $3322212$. 

Once we have chosen our preferred non semi-symmetric word $\alpha$ of odd length, we compute the Markov value $\ell$ of the periodic sequence $\dots\alpha\alpha\dots$, and we select a Cantor set $\Sigma_{\alpha}$ of sequences whose Markov values are $<\ell$.

Since $\alpha$ is \emph{not} semi-symmetric, the problems of gluing sequences in $\Sigma_{\alpha}$ on the left and/or on the right of $\dots\alpha\alpha\dots$ in such a way that the Markov value of the resulting sequence doesn't increase too much might have \emph{distinct} answers. In fact, if $\alpha$ decomposes as $\alpha=xy$, then the Markov values $\mu$ of $\dots\alpha\alpha z = \dots xyxyz$ with $z\in\Sigma_{\alpha}$ could be $\mu>\ell$ and \emph{systematically} smaller than the Markov values $\nu$ of $w\alpha\alpha\dots = wxyxy\dots$ with $w\in \Sigma_{\alpha}$ (because the gluings of $y$ and $z$ is a different problem from the gluings of $w$ and $x$). For example, if we try to glue the sequence $2121\dots\in\Sigma$ on the right of the periodic sequence $\dots\alpha\alpha\dots = \dots33222123322212\dots$  without increasing too much the Markov value of the resulting sequence, we might go for 
$$\dots33222123322212212121\dots$$ 
whose Markov value $\mu$ is $3.70969985975\dots$ On the other hand, if we try to glue $2121\dots\in\Sigma_{\alpha}$ on the left of $\dots\alpha\alpha\dots = \dots33222123322212\dots$ without increasing too much the Markov value, the best choice is 
$$\dots212121221233222123322212\dots$$
whose Markov value $\nu$ is $3.70969985982\dots$

In other words, the cost of gluing any $w\in\Sigma$ and $\alpha\alpha\dots$ is always higher than the cost of the sequence $\dots\alpha\alpha z$. Hence, the Markov value $\mu$ of $\dots\alpha\alpha z$ is likely to belong to the complement of $L$ because any attempt to modify the left side of $\dots\alpha\alpha z$ to reproduce big chunks of this sequence (in order to show that $\mu\in L$) would fail since it ends up producing a subword close to the sequence $z\alpha\alpha\dots\alpha\alpha z$ whose Markov value would be $\nu>\mu$.  




The discussion of the previous four paragraphs can be qualitatively rephrased in dynamical terms as follows\footnote{In the sequel, we assume some familiarity with basic aspects of the standard theory of hyperbolic sets (and we recommend the book \cite{KH} for all necessary details).}. 

The periodic word $\dots\alpha\alpha\dots$ provides a periodic point $p_{\alpha}\in\mathbb{R}^2$ of $G$ such that $\ell=\widetilde{f}(p_{\alpha})=\max_{n\in\mathbb{Z}} \widetilde{f}(G^n(p_{\alpha}))$. Also, a classical result of Perron asserts that the Markov value of any sequence in the Cantor set $\Sigma_{\alpha}:=\{1,2\}^{\mathbb{Z}}\subset \Sigma$ is $\leq \sqrt{12}<\ell$ and, moreover, the periodic word $\dots 2121\dots\in\Sigma_{\alpha}$ provides a periodic point $p_{21}\in\mathbb{R}^2$ of $G$ with $\sqrt{12}=\widetilde{f}(p_{21})=\max_{n\in\mathbb{Z}}\widetilde{f}(G^n(p_{21}))$. 

The problems of gluing sequences in $\Sigma_{\alpha}$ on the left and right of $\dots\alpha\alpha\dots$ have a clear dynamical meaning: it amounts to study the intersections $W^u_{loc}(\Sigma_{\alpha})\cap W^s_{loc}(p_{\alpha})$ and $W^s_{loc}(\Sigma_{\alpha})\cap W^u_{loc}(p_{\alpha})$ between the local stable and unstable sets of $\Sigma_{\alpha}$ and $\dots\alpha\alpha\dots$. 

Geometrically, the fact that $p_{\alpha}$ comes from a non semi-symmetric word $\alpha$ of odd length suggests that the local stable and unstable manifolds of $p_{\alpha}$ intersect the invariant manifolds of the subset $\Lambda_{\alpha}\subset\mathbb{R}^2$ related to $\Sigma_{\alpha}$ at \emph{distinct} heights with respect to $\widetilde{f}(x,y)=x+y$. In fact, one can show that the height $\mu$ of the point $\{q_{\alpha}\}:=W^u_{loc}(p_{\alpha})\cap W^s_{loc}(p_{21})$ is \emph{strictly smaller} than the minimal height $\nu$ of \emph{any} point $r\in W^s_{loc}(p_{\alpha})\cap W^u_{loc}(\Lambda_{\alpha})$: this situation is depicted in Figure \ref{f.1} below and it is quantitatively described in Lemma \ref{l.replication} below. 

Moreover, the $G$-orbit of $q_{\alpha}$ is \emph{locally unique} in the sense that some portion of the $G$-orbit of any point $z\in\mathbb{R}^2$ with $\sup_{n\in\mathbb{Z}}\widetilde{f}(G^n(z))$ close to $\mu$ must stay close to the first few $G$-iterates of $q_{\alpha}$: the quantitative incarnation of this fact is given by Corollary \ref{c.7} below. 

In this context, we show that the Markov value $\mu$ doesn't belong to the Lagrange spectrum $L$ by combining the previous two paragraphs. More concretely, if $\mu\in L$, say $\mu=\limsup\limits_{n\to+\infty}\widetilde{f}(G^n(z))$ for some $z\in\mathbb{R}^2$, then the local uniqueness property would say that some portion $\{G^{n_0}(z),\dots, G^{n_0+m_0}(z)\}$ of the $G$-orbit of $z$ is close to the first few $G$-iterates $\{q_{\alpha}, G(q_{\alpha}), \dots, G^{m_0}(q_{\alpha})\}$, so that $G^{n_0+m_0}(z)$ is close to $\Lambda_{\alpha}$. On the other hand, the assumption that $\mu=\limsup\limits_{n\to+\infty}\widetilde{f}(G^n(z))$ and the local uniqueness property say that there exists an instant $n_1>n_0+m_0$ such that $G^{n_1}(z)$ is again close to $q_{\alpha}$. However, this is impossible because the iterates of $G^{n_0+m_0}(z)$ would follow $W^u_{loc}(\Lambda_{\alpha})$ in their way to reach $G^{n_1}(z)$ and we know that the smallest height of the intersection between $W^s_{loc}(q_{\alpha})$ and $W^u_{loc}(\Lambda_{\alpha})$ is $\nu>\mu$: see Figure \ref{f.1} below for an illustration of this argument. 

\begin{figure}[htb!]
\begin{center}
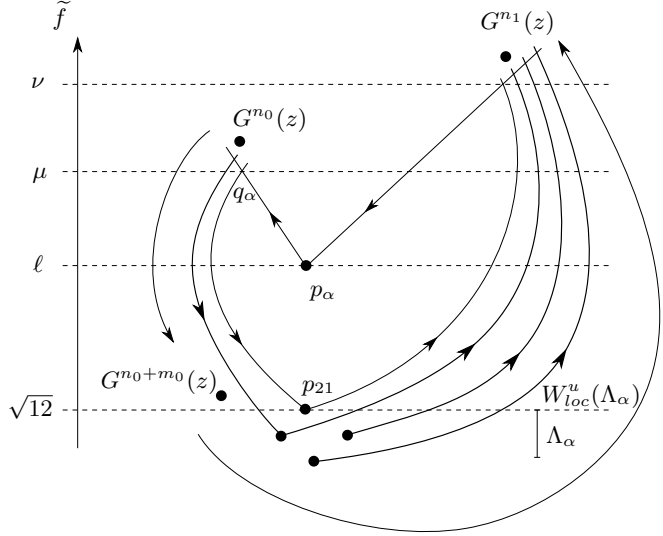\caption{Dynamics behind some elements in $M\setminus L$.}\label{f.1}
\end{center}
\end{figure}

\subsection{Ideas behind the proof of Theorem \ref{t.prrova}} 

Our proof of Theorem \ref{t.prrova} relies on the control of several portions of $M\setminus L$ in terms of the sum-set of a Cantor set associated to continued fraction expansions prescribed by a ``symmetric block'' and a Cantor set of irrational numbers whose continued fraction expansions live in the ``gaps'' of a ``symmetric block''. As it turns out, such a control is possible thanks to our key technical Lemma \ref{l.key} saying that a sufficiently large Markov value given by the sum of two continued fraction expansions systematically meeting a ``symmetric block'' must belong to the Lagrange spectrum. 

From a dynamical point of view, this argument can be qualitatively rephrased in the following way. For the sake of clarity, let us illustrate how our method\footnote{This is a non-trivial task because $HD(L\cap[\sqrt{5},\sqrt{12}])=HD(M\cap[\sqrt{5},\sqrt{12}])=1$ (cf. \cite{Mo}).} gives an estimate for $HD((M\setminus L)\cap [\sqrt{5},\sqrt{12}])$. It was shown by Hall that $HD(M\cap[\sqrt{5},\sqrt{10}])<0.93$. Hence, our task is reduced to show that $HD((M\setminus L)\cap[\sqrt{10},\sqrt{12}])<1$. Recall that Perron showed that any $\mu\in M\cap[\sqrt{5},\sqrt{12}]$ has the form $\mu=\sup_{n\in\mathbb{Z}}\widetilde{f}(G^n(z))$ for some $z\in\Lambda_{1,2}$, where $\Lambda_{1,2}\subset\mathbb{R}^2$ is the set of points related to $\{1,2\}^{\mathbb{Z}}\subset\Sigma$. Suppose that $\mu\in (M\setminus L)\cap [\sqrt{10},\sqrt{12}]$ is associated to a point $z\in\Lambda_{1,2}$ with $\mu=\widetilde{f}(z)=\sup_{n\in\mathbb{Z}} \widetilde{f}(G^n(z))$. The set of points $\Lambda_{11,22}\subset\mathbb{R}^2$ related to the sequences in $\{11,22\}^{\mathbb{Z}}\subset\Sigma$ is the geometric incarnation of a ``symmetric block'' in the sense that $\{11,22\}^{\mathbb{Z}}$ is a shift-invariant, locally maximal, \emph{transitive} set. In this situation, the $G$-orbit of $z$ can't accumulate on $\Lambda_{11,22}$ in the past and in the future. In fact, if both the $\alpha$-limit and $\omega$-limit sets of $z$ intersect $\Lambda_{11,22}$ at two points $z_{-\infty}$ and $z_{+\infty}$, then the transitivity of $\Lambda_{11,22}$ would allow us to employ the \emph{shadowing lemma} to construct a $G$-orbit $\{G^n(w)\}_{n\in\mathbb{Z}}$ tracking certain periodic pseudo-orbits starting at $z$, reaching $z_{+\infty}$, going to $z_{-\infty}$ and coming back to $z$: this scenario is qualitatively described in Figure \ref{f.2} below and its quantitative incarnation is Lemma \ref{l.key} below. In particular, $\mu=\limsup_{n\to\infty}\widetilde{f}(G^n(w))\in L$, a contradiction. Therefore, if $\mu\in(M\setminus L)\cap[\sqrt{10}, \sqrt{12}]$, then the past or the future of the $G$-orbit of $z\in\Lambda_{1,2}$ can't approach $\Lambda_{11,22}$: in other words, $\{G^n(z)\}_{n\geq 0}$ or $\{G^n(z)\}_{n\leq 0}$ travels through $\Lambda_{1,2}$ while avoiding some neighborhood of $\Lambda_{11,22}$, i.e., $\{G^n(z)\}_{n\geq 0}$ or $\{G^n(z)\}_{n\leq 0}$ lives in the ``gaps'' of $\Lambda_{11,22}$ in $\Lambda_{1,2}$. 

\begin{figure}[htb!]
\begin{center}
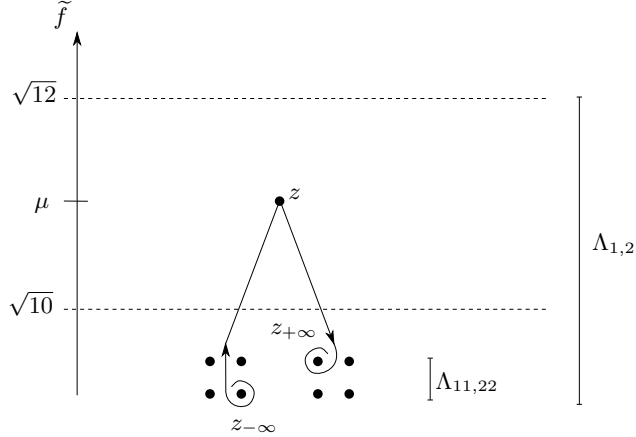\caption{Dynamical constraints on $M\setminus L$.}\label{f.2}
\end{center}
\end{figure}

An interesting feature of our method is its flexibility: we have some freedom in our choices of ``symmetric blocks''. Of course, there is a price to pay in this process: if one tries to refine the symmetric blocks to fit better the portions of $M\setminus L$, then one is obliged to estimate the Hausdorff dimension of Cantor sets of irrational numbers whose continued fraction expansions satisfy complicated restrictions. 

In our proof of Theorem \ref{t.prrova}, we chose the symmetric blocks in order to rely only on Cantor sets whose Hausdorff dimensions were \emph{rigorously} estimated by Hensley in \cite{He}. 

Nevertheless, one can get better \emph{heuristic} bounds for $HD(M\setminus L)$ thanks to the several methods in the literature to numerically approximate the Hausdorff dimension of Cantor sets of numbers with prescribed restrictions of their continued fraction expansions. By implementing the ``thermodynamical method'' introduced by Jenkinson--Pollicott in \cite{JP01}, we obtained the \emph{empirical} bound $HD(M\setminus L)<0.888$. 

\begin{remark} As it was observed by Jenkinson--Pollicott in \cite{JP16}, it is possible in principle to convert the heuristic estimates obtained with their methods into rigorous bounds. However, we will not pursue this direction here.  
\end{remark}

\subsection{Organization of the article} Closing this introduction, let us explain the organization of this paper. After recalling some basic definitions in Section \ref{s.preliminaries}, we dedicated the subsequent three section to the proof of Theorem \ref{t.1}. More precisely, we show in Section \ref{s.uniqueness} that any Markov value close to $3.70969985975025\dots$ can only be realized by a sequence containing the word 
$$\beta=2332221233222123322$$
In Section \ref{s.replication}, we show that a sequence $a$ containing $\beta$ whose Markov value is  $<3.70969985975033$ is necessarily periodic on the left, i.e., 
$$a=\overline{3322212}33222123322$$
(where $\overline{3322212}$ means an infinite concatenation of $3322212$). 

Then, we derive Theorem \ref{t.1} in Section \ref{s.results} as a consequence of a more precise result (cf. Theorem \ref{t.A} below) and the recent work of Jenkinson--Pollicott \cite{JP16}.    

Next, we devote the remainder of the text to the proof of Theorem \ref{t.prrova}. More concretely, we prove in Section \ref{s.key} our main technical result namely, Lemma \ref{l.key}. Once we dispose of this lemma in our toolbox, we employ it in Section \ref{s.prrova} to describe several portions of $M\setminus L$ (i.e., the intersections of $M\setminus L$ with the intervals $(\sqrt{10}, \sqrt{13})$, $(\sqrt{13}, 3.84)$, $(3.84, \sqrt{20})$ and $(\sqrt{20},\sqrt{21})$) as subsets of arithmetic sums of relatively explicit Cantor sets; in particular, this permits to establish Theorem \ref{t.prrova}. Finally, we show in Appendix \ref{s.prova} how a refinement of the discussion in Section \ref{s.prrova} can be combined with the Jenkinson--Pollicott algorithm to give the heuristic bound $HD(M\setminus L)<0.888$. 

\subsection*{Acknowledgments} We are thankful to Dmitry Gayfulin for his feedback on our proof of Theorem \ref{t.1}, and we are grateful to Pascal Hubert for his interest in this project. This article was partly written during a visit of the first author to IMPA (Brazil) sponsored by CAPES project 88887.136371/2017-00. The first author warmly thanks IMPA's staff for the hospitality and CAPES for the financial support. The second author is grateful to the financial support of FAPERJ and CNPq. 

\section{Some preliminaries}\label{s.preliminaries} 

Given a finite word $\gamma$, the sequence obtained by infinite concatenation of this word is denoted by $\overline{\gamma}$. 

In general, we will indicate the symbol $a_0$ at the zeroth position of a sequence $a=(a_n)_{n\in\mathbb{Z}}$ by an asterisk, i.e., $a=\dots a_{-2} a_{-1} a_0^* a_1 a_2 \dots$.  

An elementary result for comparing continued fractions is the following lemma\footnote{Compare with Lemmas 1 and 2 in Chapter 1 of Cusick-Flahive book \cite{CF}.}:

\begin{lemma}\label{l.0} Let $\alpha=[a_0; a_1,\dots, a_n, a_{n+1},\dots]$ and $\beta=[a_0; a_1,\dots, a_n, b_{n+1},\dots]$ with $a_{n+1}\neq b_{n+1}$. Then:
\begin{itemize}
\item $\alpha>\beta$ if and only if $(-1)^{n+1}(a_{n+1}-b_{n+1})>0$;
\item $|\alpha-\beta|<1/2^{n-1}$.
\end{itemize}
\end{lemma} 

The reader is encouraged to consult the book \cite{CF} by Cusick and Flahive for more background on continued fractions and their relationship to the Lagrange and Markov spectra. 

In the next three sections, we concentrate on the proof of Theorem \ref{t.1}. 

\section{Local uniqueness of candidate sequences}\label{s.uniqueness}

In this entire section, we deal exclusively with sequences $a=(a_n)_{n\in\mathbb{Z}}\in\{1,2,3\}^{\mathbb{Z}}$.

\begin{lemma}\label{l.1} 
\begin{itemize}
\item[(i)] $\lambda_0(\dots 3^*1\dots)>3.822$
\item[(ii)] $\lambda_0(\dots 2 3^* 2\dots)>3.7165$
\item[(iii)] $\lambda_0(\dots 3 3^* 3\dots)<3.61279$ 
\end{itemize}
\end{lemma}
\begin{proof}
$\lambda_0(\dots 3^*1\dots)\geq[3;1,\overline{1,3}]+[0;\overline{3,1}]=3.822020185\dots$, $\lambda_0(\dots 2 3^* 2\dots)\geq [3;2,\overline{1,3}]+[0;2,\overline{1,3}] = 3.7165151389911\dots$ and $\lambda_0(\dots 3 3^* 3\dots)\leq [3;3,\overline{3,1}] + [0;3,\overline{3,1}]=3.61278966\dots$.  
\end{proof}

An immediate corollary of this lemma is:

\begin{corollary}\label{c.1} If $3.62<\lambda_0(a)<3.71$, then $a=\dots 33^*2\dots$ up to transposition. 
\end{corollary}

\begin{lemma}\label{l.2}  
\begin{itemize}
\item[(iv)] $\lambda_0(\dots 33^*21\dots)<3.6973$
\item[(v)] $\lambda_0(\dots 3 3^* 23\dots)>3.72$ or $\lambda_{-1}(\dots 33^*23\dots)> 3.822$
\end{itemize}
\end{lemma}
\begin{proof}  $\lambda_0(\dots 33^*21\dots)\leq [3;2,1,\overline{1,3}]+[0;3,\overline{3,1}]=3.6972\dots$, and, by Lemma \ref{l.1} (i),  either  $\lambda_{-1}(\dots 33^*23\dots)> 3.822$ or $\lambda_0(\dots 3 3^* 23\dots)\geq [3;2,3,\overline{3,1}]+[0;3,2,\overline{3,1}] = 3.72\dots$.
\end{proof}

\begin{corollary}\label{c.2} If $3.698<\lambda_0(a)<3.71$ and $\lambda_i(a)<3.71$ for $|i|\leq 1$, then $a=\dots 33^*22\dots$ up to transposition.
\end{corollary}

\begin{lemma}\label{l.3} 
\begin{itemize}
\item[(vi)] $\lambda_0(\dots 333^*22\dots)>3.71$
\item[(vii)] $\lambda_i(\dots 23 3^* 221\dots) > 3.7099$ for some $i\in\{-3,0,5\}$
\item[(viii)] $\lambda_0(\dots 233^*223\dots)< 3.7087$
\item[(ix)] $\lambda_0(\dots 3233^*222\dots)\leq \lambda_0(\dots 2233^*222\dots) < 3.7084$
\end{itemize}
\end{lemma}

\begin{proof} $\lambda_0(\dots 333^*22\dots)\geq [3;2,2,\overline{3,1}]+[0;3,3,\overline{3,1}] = 3.71\dots$. 

By Lemma \ref{l.1} (i), (ii), either $\lambda_5(\dots 23 3^* 221\dots)>3.822$, or $\lambda_{-3}(\dots 23 3^* 221\dots)>3.71$, or $\lambda_0(\dots 23 3^* 221\dots)\geq [3;2,2,1,1,2,\overline{1,3}]+ [0;3,2,3,3,\overline{3,1}] = 3.7099028\dots$

$\lambda_0(\dots 233^*223\dots)\leq [3;2,2,3,\overline{3,1}]+[0;3,2,\overline{1,3}] = 3.708691\dots$

$\lambda_0(\dots 3233^*222\dots)\leq \lambda_0(\dots 2233^*222\dots)\leq [3;2,2,2,\overline{3,1}]+[0;3,2,2,\overline{3,1}] < 3.7083107\dots$
\end{proof}

\begin{corollary}\label{c.3}  If $3.7087<\lambda_0(a)<3.7099$ and $\lambda_i(a)<3.7099$ for $|i|\leq 5$, then $a=\dots 1233^*222\dots$ up to transposition.
\end{corollary}

\begin{lemma}\label{l.4}
\begin{itemize}
\item[(x)] $\lambda_i(\dots 1233^*2223\dots)>3.7099$ for some $i\in\{-5,0\}$
\item[(xi)] $\lambda_0(\dots 11233^*2221\dots) < 3.7096$
\item[(xii)] $\lambda_0(\dots 21233^*2222\dots) > 3.71$
\item[(xiii)] $\lambda_0(\dots 21233^*22211\dots) > 3.7097$
\item[(xiv)] $\lambda_0(\dots 111233^*22221\dots)\geq \lambda_0(\dots 111233^*22222\dots) > 3.7097$
\item[(xv)] $\lambda_i(\dots 111233^*22223\dots) > 3.7097$ for some $i\in\{-7,0,5\}$ 
\item[(xvi)] $\lambda_0(\dots 211233^*22223\dots)\leq \lambda_0(\dots 211233^*22222\dots) < 3.70957$ 
\item[(xvii)] $\lambda_i(\dots 121233^*22212\dots)>3.7097$ for some $i\in\{-7,0,7\}$
\item[(xviii)] $\lambda_0(\dots 321233^*22212\dots)<3.709604$
\end{itemize}
\end{lemma}

\begin{proof} By Lemma \ref{l.1} (i), if $\lambda_{-5}(\dots 1233^*2223\dots)<3.82$, then $\lambda_0(\dots 1233^*2223\dots)\geq [3;2,2,2,3,\overline{3,1}]+[0;3,2,1,1,2,\overline{1,3}]= 3.7099\dots$

$\lambda_0(\dots 11233^*2221\dots)\leq[3;2,2,2,1,\overline{1,3}] + [0;3,2,1,1,\overline{1,3}]=3.709507\dots$ 

$\lambda_0(\dots 21233^*2222\dots)\geq [3;2,2,2,2,\overline{3,1}]+[0;3,2,1,2,\overline{3,1}] = 3.7107\dots$

 $\lambda_0(\dots 21233^*22211\dots) \geq [3;2,2,2,1,1,\overline{1,3}]+[0;3,2,1,2,\overline{3,1}] = 3.7097\dots$

$\lambda_0(\dots 111233^*22222\dots)\geq [3;2,2,2,2,2,\overline{1,3}]+ [0;3,2,1,1,1,\overline{1,3}] = 3.7097\dots$

By Lemma \ref{l.1} (i), (ii), if $\lambda_i(\dots 111233^*22223\dots) < 3.7097$ for $i\in\{-7,5\}$, then $\lambda_i(\dots 111233^*22223\dots) \geq [3;2,2,2,2,3,3,\overline{3,1}]+[0;3,2,1,1,1,1,2,\overline{1,3}]= 3.7097\dots$

$\lambda_0(\dots 211233^*22222\dots)\leq[3;2,2,2,2,2,\overline{3,1}]+[0;3,2,1,1,2,\overline{3,1}]=3.709568\dots$

By Lemma \ref{l.1} (i), if $\lambda_i(\dots 121233^*22212\dots)<3.82$ for $|i|=7$, then $\lambda_0(\dots 121233^*22212\dots)\geq [3;2,2,2,1,2,1,2,\overline{1,3}]+[0;3,2,1,2,1,1,2,\overline{1,3}]=3.7097\dots$

$\lambda_0(\dots 321233^*22212\dots)\leq [3;2,2,2,1,2,\overline{3,1}]+[0;3,2,1,2,3,\overline{3,1}] = 3.709603\dots$
\end{proof}

\begin{corollary}\label{c.4} If $3.709604<\lambda_0(a)<3.7097$ and $\lambda_i(a)<3.7097$ for $|i|\leq 7$, then $a=\dots 221233^*22212\dots$ or $\dots 211233^*22221\dots$ up to transposition.
\end{corollary}

\begin{proof} By Corollary \ref{c.3}, $a=\dots 1233^*222\dots$. By Lemma \ref{l.4} (x), $a=\dots 1233^*2221\dots$ or $\dots 1233^*2222\dots$. By Lemma \ref{l.1} (i), Lemma \ref{l.4} (xi), (xii), $a=\dots 21233^*2221\dots$ or $\dots 11233^*2222\dots$ 

By Lemma \ref{l.1} (i), Lemma \ref{l.4} (xiii), (xiv), (xv), (xvi), $a=\dots 21233^*22212\dots$ or $\dots 11233^*22221\dots$ 

By Lemma \ref{l.1} (i), Lemma \ref{l.4} (xiv), (xvii), (xviii), $a=\dots 221233^*22212\dots$ or $\dots 211233^*22221\dots$
\end{proof}

\begin{lemma}\label{l.5} 
\begin{itemize}
\item[(xix)] $\lambda_0(\dots 221233^*222121\dots)<3.709642$ 
\item[(xx)] $\lambda_0(\dots 1221233^*222122\dots) \leq \lambda_0(\dots 2221233^*222122\dots) < 3.709693$
\item[(xxi)] $\lambda_0(\dots 1221233^*2221233\dots) < 3.70968$
\item[(xxii)] $\lambda_0(\dots 3221233^*2221233\dots) > 3.7097$
\item[(xxiii)] $\lambda_0(\dots 1211233^*222211\dots) \leq \lambda_0(\dots 2211233^*222211\dots) <3.70969$
\item[(xxiv)] $\lambda_0(\dots 1211233^*222212\dots) < 3.70969$  
\item[(xxv)] $\lambda_0(\dots 3211233^*222212\dots) > 3.7097$
\end{itemize}
\end{lemma}

\begin{proof} $\lambda_0(\dots 221233^*222121\dots)\leq [3;2,2,2,1,2,1,\overline{1,3}]+[0;3,2,1,2,2,\overline{3,1}] = 3.709641\dots$

$\lambda_0(\dots 2221233^*222122\dots)\leq [3;2,2,2,1,2,2,\overline{1,3}]+[0;3,2,1,2,2,2,\overline{1,3}] = 3.7096929\dots$

$\lambda_0(\dots 1221233^*2221233\dots)\leq [3;2,2,2,1,2,3,3,\overline{3,1}]+[0;3,2,1,2,2,1,\overline{1,3}] = 3.709679\dots$

$\lambda_0(\dots 3221233^*2221233\dots)\geq [3;2,2,2,1,2,3,3,\overline{1,3}]+[0;3,2,1,2,2,3,\overline{3,1}] = 3.70972\dots$

$\lambda_0(\dots 2211233^*222211\dots)\leq [3;2,2,2,2,1,1,\overline{1,3}]+[0;3,2,1,1,2,2,\overline{1,3}] = 3.709688\dots$

$\lambda_0(\dots 1211233^*222212\dots)\leq [3;2,2,2,2,1,2,\overline{1,3}]+[0;3,2,1,1,2,1,\overline{1,3}] = 3.709681\dots$

$\lambda_0(\dots 3211233^*222212\dots)\geq [3;2,2,2,2,1,2,\overline{3,1}]+[0;3,2,1,1,2,3,\overline{3,1}] = 3.70974\dots$
\end{proof}

\begin{corollary}\label{c.5} If $3.709693<\lambda_0(a)<3.7097$ and $\lambda_i(a)<3.7097$ for $|i|\leq 9$, then $a=\dots 2221233^*2221233 \dots$ or $\dots 33211233^*222211 \dots$ or $\dots 2211233^*222212\dots$ up to transposition.
\end{corollary}

\begin{proof} By Corollary \ref{c.4}, $a=\dots 221233^*22212\dots$ or $\dots 211233^*22221\dots$. By Lemma \ref{l.5} (xix) and Lemma \ref{l.1} (i), (ii), the sole possible extensions $a$ are 
$$\dots 221233^*222122\dots \quad \quad \dots 221233^*2221233\dots$$
$$\dots 211233^*222211\dots \quad \quad \dots 211233^*222212\dots$$
 
By Lemma \ref{l.5} (xx), (xxi), (xxii), (xxiii), (xxiv), (xxv), the sole possible continuations of $a$ are 
$$\dots 3221233^*222122\dots \quad \quad \dots 2221233^*2221233\dots$$
$$\dots 3211233^*222211\dots \quad \quad \dots 2211233^*222212\dots$$

By Lemma \ref{l.1} (i), (ii) and Lemma \ref{l.3} (vi), we obtain that $a$ is one of the words 
$$\dots 233221233^*222122\dots \quad \quad \dots 2221233^*2221233\dots$$
$$\dots 33211233^*222211\dots \quad \quad \dots 2211233^*222212\dots$$

However, Lemma \ref{l.3} (vii) forbids the word $\dots 233221233^*222122\dots$, so that we end up with the following three possibilities
$$\dots 2221233^*2221233\dots$$
$$\dots 33211233^*222211\dots \quad \quad \dots 2211233^*222212\dots$$
for $a$.
\end{proof}

\begin{lemma}\label{l.6}
\begin{itemize}
\item[(xxvi)] $\lambda_i(\dots2221233^*22212333\dots)>3.7097$ for some $i\in\{-7,0\}$
\item[(xxvii)] $\lambda_0(\dots 33211233^*2222112\dots)<3.70968$
\item[(xxviii)] $\lambda_0(\dots 2211233^*2222121\dots) > 3.7097$
\item[(xxix)] $\lambda_i(\dots 2211233^*2222122\dots) > 3.7097$ for some $i\in\{-7,0\}$
\end{itemize} 
\end{lemma}

\begin{proof} By Lemma \ref{l.1} (i), (ii), if $\lambda_{-7}(\dots2221233^*22212333\dots)<3.71$, then $\lambda_0(\dots2221233^*22212333\dots) \geq [3;2,2,2,1,2,3,3,3,\overline{3,1}] + [0;3,2,1,2,2,2,3,3,\overline{3,1}] = 3.7097001\dots$ 

$\lambda_0(\dots 33211233^*2222112\dots)\leq[3;2,2,2,2,1,1,2,\overline{3,1}] + [0;3,2,1,1,2,3,3,\overline{3,1}] = 3.709672\dots$ 

$\lambda_0(\dots2211233^*2222121\dots)\geq[3;2,2,2,2,1,2,1,\overline{1,3}] + [0;3,2,1,1,2,2,\overline{3,1}] = 3.709711\dots$

By Lemma \ref{l.1} (i), (ii), if $\lambda_{-7}(\dots 2211233^*2222122\dots)<3.71$, then $\lambda_0(\dots2211233^*2222122\dots)\geq[3;2,2,2,2,1,2,2,\overline{1,3}] + [0;3,2,1,1,2,2,3,3,\overline{3,1}] = 3.709702\dots$
\end{proof}

\begin{corollary}\label{c.6} If $3.709698<\lambda_0(a)<3.7097$ and $\lambda_i(a)<3.7097$ for $|i|\leq 9$, then $a=\dots 2221233^*22212332 \dots$ or $\dots 33211233^*2222111 \dots$ or $\dots 2211233^*22221233\dots$ up to transposition.
\end{corollary}

\begin{proof} By Corollary \ref{c.5}, $a=\dots 2221233^*2221233 \dots$ or $\dots 33211233^*222211 \dots$ or $\dots 2211233^*222212\dots$. By Lemma \ref{l.1} (i), (ii) and Lemma \ref{l.6}, we see that the sole possible continuations of these words are $a=\dots 2221233^*22212332 \dots$ or $\dots 33211233^*2222111 \dots$ or $\dots 2211233^*22221233\dots$
\end{proof}

\begin{lemma}\label{l.7}
\begin{itemize}
\item[(xxx)] $\lambda_0(\dots12221233^*22212332\dots)>\lambda_0(\dots22221233^*22212332\dots)>3.7097$ 
\item[(xxxi)] $\lambda_0(\dots 12211233^*22221233\dots)>\lambda_0(\dots 22211233^*22221233\dots)>3.7097$
\item[(xxxii)]  $\lambda_0(\dots 2332211233^*222212333\dots)< \lambda_0(\dots 2332211233^*222212332\dots)<3.7096992$
\item[(xxxiii)] $\lambda_0(\dots2332221233^*222123321\dots)>3.7096999$
\item[(xxxiv)] $\lambda_0(\dots 233211233^*22221111\dots)<\lambda_0(\dots 333211233^*22221111\dots)<3.709696$
\item[(xxxv)] $\lambda_0(\dots 333211233^*22221112\dots)>\lambda_0(\dots 233211233^*22221112\dots)>3.7097$
\end{itemize}
\end{lemma}

\begin{proof} $\lambda_0(\dots22221233^*22212332\dots)\geq [3;2,2,2,1,2,3,3,2,\overline{3,1}]+[0;3,2,1,2,2,2,2,\overline{1,3}] = 3.709701\dots$

$\lambda_0(\dots 22211233^*22221233\dots) \geq [3;2,2,2,2,1,2,3,3,\overline{3,1}] + [0;3,2,1,1,2,2,2,\overline{1,3}] = 3.709702\dots$

$\lambda_0(\dots 2332211233^*222212332\dots) \leq [3;2,2,2,2,1,2,3,3,2,\overline{3,1}] + [0;3,2,1,1,2,2,3,3,2,\overline{3,1}]= 3.70969913\dots$

$\lambda_0(\dots2332221233^*222123321\dots)\geq[3;2,2,2,1,2,3,3,2,1,\overline{1,3}] + [0;3,2,1,2,2,2,3,3,2,\overline{1,3}] = 3.70969992\dots$

$\lambda_0(\dots 333211233^*22221111\dots)\leq [3;2,2,2,2,1,1,1,1,\overline{1,3}] + [0;3,2,1,1,2,3,3,3,\overline{1,3}] = 3.7096955\dots$

$\lambda_0(\dots 233211233^*22221112\dots)\geq [3;2,2,2,2,1,1,1,2,\overline{3,1}] + [0;3,2,1,1,2,3,3,2,\overline{3,1}] = 3.7097004\dots$
\end{proof}

\begin{corollary}\label{c.7} If $3.7096992<\lambda_0(a)<3.7096999$ and $\lambda_i(a)<3.7096999$ for $|i|\leq 9$, then 
$$a=\dots 2332221233^*222123322 \dots$$ up to transposition.
\end{corollary}

\begin{proof} By Corollary \ref{c.6}, $a$ is one of the words $\dots 2221233^*22212332 \dots$ or $\dots 33211233^*2222111 \dots$ or $\dots 2211233^*22221233\dots$

By Lemma \ref{l.7} (xxx), (xxxi), Lemma \ref{l.1} (i), (ii), and Lemma \ref{l.3} (vi), the sole possible continuations for these words are 
$$\dots 2332221233^*22212332 \dots \quad \quad \dots 233211233^*2222111 \dots$$ 
$$\dots 333211233^*2222111 \dots \quad \quad \dots 2332211233^*22221233\dots$$

However, Lemma \ref{l.1} (i) and Lemma \ref{l.7} (xxxii), (xxxiv), (xxxv) rules out all possibilities except for 
$$a = \dots 2332221233^*22212332 \dots$$

Finally, Lemma \ref{l.7} (xxxiii) and Lemma \ref{l.2} (v), this word is forced to extend as 
$$a = \dots 2332221233^*222123322 \dots$$
\end{proof}

\section{Replication mechanism}\label{s.replication}

In this entire section, we also deal exclusively with sequences $a=(a_n)_{n\in\mathbb{Z}}\in\{1,2,3\}^{\mathbb{Z}}$.

\begin{lemma}\label{l.8}
\begin{itemize}
\item[(xxxvi)] $\lambda_i(\dots22332221233^*222123322\dots)>3.70969986$ for some $i\in\{0,7\}$
\item[(xxxvii)] $\lambda_0(\dots12332221233^*2221233223\dots)>3.70969986$
\item[(xxxviii)] $\lambda_0(\dots112332221233^*222123322212\dots)>3.70969986$ 
\item[(xxxix)] $\lambda_0(\dots3212332221233^*222123322212 \dots)>3.7096998599$ 
\end{itemize} 
\end{lemma}

\begin{proof} By Lemma \ref{l.3} (vii), if $\lambda_7(\dots22332221233^*222123322\dots)<3.7099$, then 
\begin{eqnarray*}\lambda_0(\dots22332221233^*222123322\dots) &\geq& [3;2,2,2,1,2,3,3,2,2,2,\overline{3,1}] \\ &+& [0;3,2,1,2,2,2,3,3,2,2,\overline{3,1}] \\ &=& 3.70969986\dots
\end{eqnarray*}

\begin{eqnarray*}
\lambda_0(\dots12332221233^*2221233223\dots)&\geq& [3;2,2,2,1,2,3,3,2,2,3,\overline{3,1}] \\ &+& [0;3,2,1,2,2,2,3,3,2,1,\overline{3,1}] \\&=& 3.70969986\dots
\end{eqnarray*}

\begin{eqnarray*}
\lambda_0(\dots112332221233^*222123322212\dots)&\geq& [3;2,2,2,1,2,3,3,2,2,2,1,2,\overline{3,1}] \\ &+& [0;3,2,1,2,2,2,3,3,2,1,1,\overline{1,3}] \\&=& 3.70969986\dots
\end{eqnarray*}

\begin{eqnarray*}
\lambda_0(\dots3212332221233^*222123322212\dots)&\geq& [3;2,2,2,1,2,3,3,2,2,2,1,2,\overline{3,1}] \\ &+& [0;3,2,1,2,2,2,3,3,2,1,2,3,\overline{3,1}] \\&=& 3.7096998599\dots
\end{eqnarray*}
\end{proof}

\begin{lemma}\label{l.replication} $\lambda_i(\dots12212332221233^*222123322212\dots)>3.70969985975033$ for some $i\in\{-17,-15,0,13, 15\}$
\end{lemma}

\begin{proof} By Lemma \ref{l.1} (i), (ii), if $\lambda_i(\dots12212332221233^*222123322212\dots)<3.71$ for $i\in\{-17,-15,13, 15\}$, then 
\begin{eqnarray*}\lambda_i(\dots12212332221233^*222123322212\dots) &\geq& [3;2,2,2,1,2,3,3,2,2,2,1,2,3,3,3,2,\overline{3,1}] \\ &+& [0;3,2,1,2,2,2,3,3,2,1,2,2,1,1,2,1,2,\overline{1,3}] \\ &=& 3.70969985975033\dots
\end{eqnarray*}
\end{proof}

\begin{corollary}\label{c.replication} Let $a=\dots 2332221233^*222123322 \dots$ where the asterisk indicates the position $j\in\mathbb{Z}$. If $\lambda_i(a)<3.70969985975033$ for all $|i-j|\leq 17$, then 
$$a=\dots 23322212332221233^*222123322212\dots$$
and the vicinity of the position $j-7$ is $\dots2332221233^*222123322\dots$. 
\end{corollary}

\begin{proof} By Lemma \ref{l.2} (v) and Lemma \ref{l.8} (xxxvi), our word must extends as 
$$a=\dots12332221233^*222123322\dots$$

By Lemma \ref{l.3} (vii) and Lemma \ref{l.8} (xxxvii), our word is forced to continue as 
$$a=\dots12332221233^*2221233222\dots$$

By Lemma \ref{l.1} (i), we have the following possibilities 
$$\dots112332221233^*2221233222\dots \quad \textrm{or} \quad \dots212332221233^*2221233222\dots$$
for the word $a$. 

By Lemma \ref{l.4} (x), (xii), these two words can continue only as 
$$\dots112332221233^*22212332221\dots \quad \textrm{or} \quad  \dots212332221233^*22212332221\dots$$

By Lemma \ref{l.4} (xiii), these words are obliged to extend as 
$$\dots112332221233^*222123322212\dots \quad \textrm{or} \quad  \dots212332221233^*222123322212\dots$$ 

However, we can apply Lemma \ref{l.8} (xxxviii) to rule out the first case, so that 
$$a=\dots212332221233^*222123322212\dots$$

By Lemma \ref{l.4} (xvii) and Lemma \ref{l.8} (xxxix), this word continues as 
$$a=\dots2212332221233^*222123322212\dots$$  

By Lemma \ref{l.replication} and Lemma \ref{l.5} (xxii), we have to extend as 
$$a=\dots22212332221233^*222123322212\dots$$

By Lemma \ref{l.7} (xxx), we are forced to continue as 
$$a=\dots322212332221233^*222123322212\dots$$

Finally, Lemma \ref{l.1} (i), (ii) and Lemma \ref{l.3} (vi) reveal that 
$$a=\dots23322212332221233^*222123322212\dots$$
\end{proof}

\section{Lower bound on the Hausdorff dimension of $M\setminus L$}\label{s.results}

\begin{proposition}\label{p.1} $L\cap (3.70969985968, 3.70969985975033) = \emptyset$
\end{proposition}

\begin{proof} Suppose that $\ell\in L\cap (3.70969985968, 3.70969985975033)$ and let $a\in\{1,2,3\}^{\mathbb{Z}}$ be a sequence such that $\ell=\limsup\limits_{n\to\infty}\lambda_n(a)$. By repeatedly applying Corollaries \ref{c.7} and \ref{c.replication}, we would deduce that 
$$\ell=\lambda_0(\overline{33^*22212})= 3.709699859679\dots < 3.70969985968$$ 
a contradiction. 
\end{proof}

\begin{proposition}\label{p.2}$C=\{\lambda_0(\overline{3322212}33^*2221233222122121212\theta): \theta\in\{1,2\}^{\mathbb{N}}\}$ is contained in $M\cap(3.70969985975024, 3.70969985975028)$.
\end{proposition}
\begin{proof} This is a straightforward calculation. 
\end{proof}

The previous two propositions imply that:  

\begin{theorem}\label{t.A} $C=\{\lambda_0(\overline{3322212}33^*2221233222122121212\theta): \theta\in\{1,2\}^{\mathbb{N}}\}\subset M\setminus L$. 
\end{theorem}

By Jenkinson--Pollicott work \cite{JP16}, a consequence of Theorem \ref{t.A} is the following slightly stronger version of Theorem \ref{t.1}: 

\begin{corollary} The Hausdorff dimension of $(M\setminus L)\cap (3.7096, 3.7097)$ is 
$$\geq 0.5312805062772051416244686473684717854930591090183\dots$$
\end{corollary}

In Appendix \ref{s.Cantor-Cusick}, we refine our proof of Theorem \ref{t.1} to determine the largest interval $J$ containing the set $C$ from Proposition \ref{p.2} with $J\cap L=\emptyset$, to compute the largest element $\Upsilon$ of $(M\setminus L)\cap J$, and to exhibit a Cantor set $\Omega$ of continued fraction expansions such that $HD(\Omega) = HD((M\setminus L)\cap J)$. 

For now, we consider that the discussion of Theorem \ref{t.1} is complete and we move on to the discussion of the proof of Theorem \ref{t.prrova}. 

\section{Key lemma towards Theorem \ref{t.prrova}}\label{s.key}

Denote by 
$$K(A) = \{[0;\gamma]:\gamma\in\Sigma^+(A)\} \quad \textrm{and} \quad K^-(A) = \{[0;\delta^t]:\delta\in\Sigma^-(A)\}$$ 
the Cantor sets of the real line naturally associated to a subshift of finite type $\Sigma(A)\subset \Sigma$. 

Fix $\Sigma(B)\subset\Sigma(C)$ two transitive subshifts of finite type of $\Sigma=(\mathbb{N}^*)^{\mathbb{Z}}$ such that $K(B)=K^-(B)$ and $K(C) = K^-(C)$, i.e., $B$ and $C$ are \emph{symmetric}. 

Denote by $\underline{a}=(a_n)_{n\in\mathbb{Z}}\in\Sigma(C)$ be a sequence with $m(\underline{a}) = f(\underline{a}) = m\in M$ and   
\begin{equation}\label{e.condition-m} 
m>\max\limits_{\substack{\beta\in \Sigma^+(B) \\ \alpha\in \Sigma^-(C) \\ \eta \textrm{ predecessor in }\Sigma(B) \textrm{ of }\beta}} \left([0;\beta]^{-1}+[0;(\alpha\eta)^t]\right):=c(B,C)
\end{equation}

\begin{lemma}\label{l.key} Suppose that, for every $k\in\mathbb{N}$, there exists $n_k,m_k\geq k$ such that: 
\begin{itemize}
\item[(i)] the half-infinite sequence $\dots a_0^*\dots a_{n_k}$ can be completed into two bi-infinite sequences $\underline{\theta}_k^{(1)}=\dots a_0^*\dots a_{n_k}\underline{\alpha_k}$ and $\underline{\theta}_k^{(2)}=\dots a_0^*\dots a_{n_k}\underline{\beta_k}$ so that $K(B)\cap [[0;\underline{\alpha_k}], [0;\underline{\beta_k}]]\neq\emptyset$;
\item[(ii)] the half-infinite sequence $a_{-m_k}\dots a_0^*\dots$ can be completed into two bi-infinite sequences $\underline{\theta}_k^{(3)}=\underline{\gamma_k}a_{-m_k}\dots a_0^*\dots$ and $\underline{\theta}_k^{(4)} = \underline{\delta_k}a_{-m_k}\dots a_0^*\dots$ so that $K(B)\cap [[0;\underline{\gamma_k}^t], [0;\underline{\delta_k}^t]]\neq\emptyset$; 
\item[(iii)] $\lim\limits_{k\to\infty}m(\underline{\theta}_k^{(j)}) = m$ for each $1\leq j\leq 4$.
\end{itemize} 
Then, $m\in L$. 
\end{lemma}

\begin{proof} By Theorem 2 in Chapter 3 of Cusick--Flahive book \cite{CF}, our task is reduced to show that $m=\lim\limits_{k\to\infty}m(P_k)$ where $P_k$ is a sequence of periodic points in $\Sigma$.  

For each $k\in\mathbb{N}$, let us take $\underline{\mu_k}, \underline{\nu_k}\in\Sigma(B)$ with 
\begin{equation}\label{e.1}
[0;\underline{\alpha_k}]\leq [0;\pi^+(\underline{\mu_k})]\leq [0;\underline{\beta_k}] \quad \textrm{ and } \quad [0;\underline{\gamma_k}^t]\leq [0;(\pi^+(\underline{\nu_k}))^t]\leq [0;\underline{\delta_k}^t]
\end{equation} 
The transitivity of the subshift of finite type $\Sigma(B)$ permits to choose finite subword $(\mu\ast\nu)_k$ of an element of $\Sigma(B)$ connecting the initial segment $\mu_k^+$ of $\pi^+(\underline{\mu_k})$ of length $k$ with the final segment $\nu_k^-$ of $\pi^-(\underline{\nu_k})$ of length $k$, say 
$$(\mu\ast\nu)_k = \mu_k^+ w_k \nu_k^-$$

In this setting, consider the periodic point $P_k\in \Sigma$ obtained by infinite concatenation of the finite block 
$$a_0^*\dots a_{n_k}(\mu\ast\nu)_k a_{-m_k}\dots a_{-1}$$

By \eqref{e.1} and Lemma \ref{l.0}, for each $-1-m_k\leq j\leq n_k+1$, one has  
\begin{equation}\label{e.2}
f(\sigma^j(P_k))\leq \max\{m(\underline{\theta}_k^{(1)}), m(\underline{\theta}_k^{(2)}), m(\underline{\theta}_k^{(3)}), m(\underline{\theta}_k^{(4)})\} + \frac{1}{2^{k-2}}
\end{equation}

By definition of $c(B,C)$, the facts that $\Sigma(B)\subset \Sigma(C)$ and $B, C$ are symmetric, and Lemma \ref{l.0}, one has  
\begin{equation}\label{e.3}
f(\sigma^j(P_k))<c(B,C)+\frac{1}{2^{k-1}}
\end{equation}
for each $j$ corresponding to a non-extremal position in $(\mu\ast\nu)_k$.  

Also, by Lemma \ref{l.0}, we know that 
\begin{equation}\label{e.4}
f(P_k)>f(\underline{a})-\frac{1}{2^{k-2}} = m-\frac{1}{2^{k-2}}
\end{equation}

It follows from \eqref{e.condition-m}, \eqref{e.2}, \eqref{e.3} and \eqref{e.4} that 
$$m-\frac{1}{2^{k-2}}\leq m(P_k)\leq \max\{m(\underline{\theta}_k^{(1)}), m(\underline{\theta}_k^{(2)}), m(\underline{\theta}_k^{(3)}), m(\underline{\theta}_k^{(4)})\} + \frac{1}{2^{k-2}}$$ 
for all $k$ sufficiently large. In particular, $m=\lim\limits_{k\to\infty} m(P_k)$ thanks to our assumption in item (iii). This completes the argument. 
\end{proof}

\begin{remark} As it can be seen from the proof of this lemma, the hypothesis that $B$ and $C$ are symmetric can be relaxed to $K(B)\cap K^-(B)\neq \emptyset$ and/or $K(C)\cap K^-(C)\neq \emptyset$ after replacing \eqref{e.condition-m} by appropriate conditions on $m$. 
\end{remark}

\section{Rigorous estimates for $HD(M\setminus L)$}\label{s.prrova}

In this section, we use Hensley's estimates in \cite{He} for the Hausdorff dimensions of the Cantor sets $K(\{1,2\})$, $K(\{1,2,3\})$ and $K(\{1,2,3,4\})$ in order to \emph{rigorously} prove Theorem \ref{t.prrova}.

\subsection{Description of $(M\setminus L)\cap(\sqrt{10},\sqrt{13})$}\label{ss.basic-argument}

Let $m\in M\setminus L$ with $\sqrt{10}<m<\sqrt{13}$. In this setting, $m=m(\underline{a})=f(\underline{a})$ for a sequence $\underline{a}=(\dots, a_{-1}, a_0, a_1, \dots)\in \{1,2\}^{\mathbb{Z}}=:\Sigma(C)$ (see, e.g., Lemma 7 in Chapter 1 of Cusick--Flahive book \cite{CF}). 

Consider the (complete) subshift $\Sigma(B)=\{11,22\}^{\mathbb{Z}}\subset\Sigma(C)$. Note that $B$ and $C$ are symmetric, and the quantity $c(B,C)$ introduced above is bounded by 
\begin{equation}\label{e.c-sqrt12}
c(B,C) \leq \max\limits_{\substack{\beta\in K(B) \\ \alpha\in K(C)}} (\beta^{-1}+\alpha) = [2;\overline{2,2}]+[0;\overline{1,2}] = \sqrt{2} + \sqrt{3} <\sqrt{10} < m
\end{equation} 
thanks to Lemma \ref{l.0}. 

Fix $\varepsilon>0$ such that $[m-2\varepsilon, m+2\varepsilon]\cap L=\emptyset$ and take $N\in\mathbb{N}$ with $f(\sigma^j(\underline{a})) < m-2\varepsilon$ for all $|j|\geq N$.

For each $n\in\mathbb{N}^*$, resp. $-n\in\mathbb{N}^*$, let us consider the possible continuations of $\dots a_0^*\dots a_n$, resp. $a_{n}\dots a_0^*\dots$ into sequences in $\Sigma(C)$ whose Markov values are attained in a position $|j|\leq N$. 

Of course, for every $n\in\mathbb{Z}\setminus\{0\}$, we have the following cases:
\begin{itemize} 
\item[(a)] there is an unique continuation (prescribed by $\underline{a}$); 
\item[(b)] there are two distinct continuations given by half-infinite sequences $\alpha_n$ and $\beta_n$; in this context, one has two subcases:
\begin{itemize}
\item[(b1)] the interval $I_n$ determined by $[0;\alpha_{n}]$ and $[0;\beta_{n}]$, resp. $[0; (\alpha_n)^t]$ and $[0; (\beta_n)^t]$, when $n>0$, resp. $n<0$, is \emph{disjoint} from $K(B)$; 
\item[(b2)] the interval $I_n$ determined by $[0;\alpha_{n}]$ and $[0;\beta_{n}]$, resp. $[0; (\alpha_n)^t]$ and $[0; (\beta_n)^t]$, when $n>0$, resp. $n<0$, \emph{intersects} $K(B)$;
\end{itemize}
\end{itemize} 

\begin{proposition}\label{p.M-L-constraint} There exists $k\in\mathbb{N}$ such that: 
\begin{itemize}
\item either for all $n\geq k$ the subcase (b2) doesn't occur;
\item or for all $n\leq -k$ the subcase (b2) doesn't occur. 
\end{itemize}
\end{proposition}

\begin{proof} If there were two subsequences $n_k, m_k\to+\infty$ so that the case (b2) happens for all $n_k$ and $-m_k$, then \eqref{e.c-sqrt12} and Lemma \ref{l.key} would say that $m\in L$, a contradiction with our assumption $m\in M\setminus L$. 
\end{proof} 

The previous proposition says that, up to replacing $\underline{a}$ by its transpose, there exists $k\geq 0$ such that, for all $n\geq k$, either $\dots a_0^*\dots a_n$ has a forced continuation $\dots a_0^*\dots a_n a_{n+1}\dots$ or two continuations $\dots a_0^*\dots a_n\alpha_n$ and $\dots a_0^*\dots a_n\beta_n$ with $[[0;\alpha_n],[0;\beta_n]]\cap K(B)=\emptyset$. Therefore, we conclude that in this setting 
\begin{equation}\label{e.M-L-constraint}
m=a_0 + [0;a_1,\dots]+[0;a_{-1},\dots]
\end{equation}
where $[0;a_{-1},\dots]\in K(C)$ and $[0;a_1,\dots]$ belongs to a set $K$ consisting of the union of a countable set $\mathcal{C}$ corresponding to the forced continuations of finite strings and a countable union of Cantor sets related to sequences generating two continued fractions at the extremities of an interval avoiding $K(B)$. 

Let $[0;a_1,\dots]\in K\setminus\mathcal{C}$ such that, for all $n$ sufficiently large, $a_1\dots a_n$ admits two continuations generating an interval avoiding $K(B)$. Given an arbitrary finite string $(b_1,\dots, b_n)$, consider the interval $I(b_1,\dots,b_n):=\{[0;b_1,\dots, b_n,\rho]: \rho> 1\}$. Recall that the length of $I(b_1,\dots,b_n)$ is 
$$|I(b_1,\dots,b_n)|=\frac{1}{q_n(q_n+q_{n-1})},$$ 
where $q_j$ is the denominator of $[0;b_1,\dots,b_j]$. We claim that the intervals $I(a_1,\dots, a_n)$ can be used to \emph{efficiently} cover $K\setminus\mathcal{C}$ as $n$ goes to infinity. For this sake, observe that if  $a_1\dots a_n$ has two continuations, say $a_1\dots a_n 1\alpha_{n+1}$ and $a_1\dots a_n 2\beta_{n+1}$ such that $[[0;2\beta_{n+1}], [0;1\alpha_{n+1}]]$ is disjoint from $K(B)$, then $\alpha_n=1\alpha_{n+1}$ and $\beta_n=2\beta_{n+1}$ start by 
$$\alpha_n = 112\alpha_{n+3} \quad \textrm{and} \quad \beta_n = 221\beta_{n+3}$$ 
In particular, we can \emph{refine} the cover of $K\setminus \mathcal{C}$ with the family of intervals  $I(a_1,\dots,a_n)$ by replacing each of them by $I(a_1,\dots,a_n,1,1,2)$ and $I(a_1,\dots,a_n,2,2,1)$. 

We affirm that this procedure doesn't increase the $(0.174813)$-Hausdorff measure of $K\setminus\mathcal{C}$. For this sake, it suffices to prove that 
\begin{equation}\label{e.subdivision-sqrt12}
|I(a_1,\dots,a_n,1,1,2)|^{s} + |I(a_1,\dots,a_n,2,2,1)|^{s} \leq |I(a_1,\dots,a_n)|^{s}
\end{equation}
with $s=0.174813$. 

In this direction, set 
$$g(s) := \frac{|I(a_1,\dots,a_n,1,1,2)|^s+|I(a_1,\dots,a_n,2,2,1)|^s}{|I(a_1,\dots,a_n)|^s}$$
The recurrence formula $q_{j+2}=a_{j+2}q_{j+1}+q_j$ implies that 
$$g(s) = \left(\frac{r+1}{(3r+5)(4r+7)}\right)^s + \left(\frac{r+1}{(3r+7)(5r+12)}\right)^s$$ 
where $r=q_{n-1}/q_n\in (0,1)$. Since $\frac{r+1}{(3r+5)(4r+7)}\leq \frac{1}{35}$ and $\frac{r+1}{(3r+7)(5r+12)}<\frac{1}{81.98}$ for all $0\leq r\leq 1$, we have  
$$g(s)<\left(\frac{1}{35}\right)^s + \left(\frac{1}{81.98}\right)^s$$
This proves \eqref{e.subdivision-sqrt12} because $\left(\frac{1}{35}\right)^{0.174813} + \left(\frac{1}{81.98}\right)^{0.174813} < 1$.

We summarize the discussion of the previous paragraphs in the following proposition:

\begin{proposition}\label{p.gap-Cantor-sqrt12} $(M\setminus L)\cap(\sqrt{10}, \sqrt{13})\subset K(\{1,2\})+K$ where $K$ is a set of Hausdorff dimension $HD(K)<0.174813$. 
\end{proposition}

An immediate corollary of this proposition is:

\begin{corollary}\label{c.M-L-sqrt10-sqrt13} $HD((M\setminus L)\cap(\sqrt{10},\sqrt{13})) < 0.706104$.
\end{corollary}

\begin{proof} By Proposition \ref{p.gap-Cantor-sqrt12} and Hensley's estimate \cite{He} for $HD(K(\{1,2\}))$, one has 
\begin{eqnarray*}
HD((M\setminus L)\cap(\sqrt{10},\sqrt{13})) &\leq& HD(K(\{1,2\}))+HD(K) \\ 
&<& 0.531291 + 0.174813 = 0.706104
\end{eqnarray*}
This proves the corollary. 
\end{proof}

\subsection{Description of $(M\setminus L)\cap(\sqrt{13}, 3.84)$}\label{ss.sqrt13-3.84} 

Let $m\in M\setminus L$ with $\sqrt{13}<m<3.84$. In this setting, $m=m(\underline{a})=f(\underline{a})$ for a sequence $\underline{a}=(\dots, a_{-1}, a_0, a_1, \dots)\in \{1,2,3\}^{\mathbb{Z}}=:\Sigma(C)$ not containing $13$ nor $31$ because if $\underline{\theta}$ contains $13$ or $31$, then a result of Bumby (explained in Table 2 in Chapter 5 of Cusick--Flahive's book \cite{CF}) implies that $m(\underline{\theta})>3.84$.

Consider the (complete) subshift $\Sigma(B)=\{1,2\}^{\mathbb{Z}}\subset\Sigma(C)$. Note that $B$ and $C$ are symmetric, and the quantity $c(B,C)$ introduced above is bounded by 
\begin{equation*}
c(B,C) \leq \max\limits_{\substack{\beta\in K(B) \\ \alpha\in K(C)}} (\beta^{-1}+\alpha) = [2;\overline{1,2}]+[0;\overline{1,3}] < \sqrt{13} < m
\end{equation*} 
thanks to Lemma \ref{l.0}. 

We proceed similarly to Subsection \ref{ss.basic-argument}. More precisely, the same arguments (based on Lemma \ref{l.key}) above give that, up to transposing $\underline{a}$, there exists $k\in\mathbb{N}$ such that, for all $n\geq k$, either $\dots a_0^*\dots a_n$ has a forced continuation $\dots a_0^*\dots a_n a_{n+1}\dots$ or two continuations $\dots a_0^*\dots a_n\alpha_n$ and $\dots a_0^*\dots a_n\beta_n$ with $[[0;\alpha_n],[0;\beta_n]]\cap K(B)=\emptyset$. We want to use this information to efficiently cover $(M\setminus L)\cap(\sqrt{13}, 3.84)$. For this sake, let us note that the constraint $[[0;\alpha_n],[0;\beta_n]]\cap K(B)=\emptyset$ impose severe restrictions on the possible continuations $\alpha_n$ and $\beta_n$. In particular, they fall into two types:
\begin{itemize}
\item $\alpha_n=3\alpha_{n+1}$ and $\beta_n=21\beta_{n+2}$;
\item $\alpha_n\in\{221\alpha_{n+3},23\alpha_{n+2}\}$ and $\beta_n\in\{1121\beta_{n+4},113\beta_{n+3}\}$.
\end{itemize}
Moreover, our assumption that $\underline{a}$ doesn't contain $13$ or $31$ (due to the hypothesis $m(\underline{a})=m<3.84$) says that we can \emph{ignore} the case $\beta_n\in\{113\beta_{n+3}\}$. 

In summary, we have that the $s$-Hausdorff measure of the set 
$$K:=\{\underline{a}=[a_0;a_1,\dots]: \sqrt{13} < m(\underline{a}) < 3.84\}$$ is finite for any parameter $s$ with 
$$g(s)=\frac{|I(a_1,\dots,a_n,3)|^s + |I(a_1,\dots,a_n,2,1)|^s}{|I(a_1,\dots,a_n)|^s}<1$$
and 
$$h(s)=\frac{|I(a_1,\dots,a_n,2,2,1)|^s + |I(a_1,\dots,a_n,2,3)|^s+|I(a_1,\dots,a_n,1,1,2,1)|^s}{|I(a_1,\dots,a_n)|^s}<1$$ 
for all $(a_1,\dots,a_n)\in\bigcup\limits_{k\in\mathbb{N}}\{1,2,3\}^k$. 

The recurrence $q_{j+2}=a_{j+2}q_{j+1}+q_j$ implies that 
$$\frac{|I(a_1,\dots,a_n,3)|}{|I(a_1,\dots,a_n)|}=\frac{r+1}{(r+3)(r+4)}, \quad \frac{|I(a_1,\dots,a_n,2,1)|}{|I(a_1,\dots,a_n)|}= \frac{r+1}{(r+3)(2r+5)}$$ 
and 
\begin{eqnarray*}
& \frac{|I(a_1,\dots,a_n,2,2,1)|}{|I(a_1,\dots,a_n)|} =  \frac{r+1}{(3r+7)(5r+12)}, \frac{|I(a_1,\dots,a_n,2,3)|}{|I(a_1,\dots,a_n)|} = \frac{r+1}{(3r+7)(4r+9)}, \\ & \frac{|I(a_1,\dots,a_n,1,1,2,1)|}{|I(a_1,\dots,a_n)|}= \frac{r+1}{(4r+7)(7r+12)}
\end{eqnarray*}
where $r=q_{n-1}/q_n\in(0,1)$. 

Since $\frac{r+1}{(r+3)(r+4)}\leq \frac{1}{10}$, $\frac{r+1}{(r+3)(2r+5)}\leq 0.071797$, $\frac{r+1}{(3r+7)(5r+12)}\leq 0.012197$, $\frac{r+1}{(3r+7)(4r+9)}\leq 0.016134$ and $\frac{r+1}{(4r+7)(7r+12)}\leq \frac{1}{84}$ for all $0\leq r\leq 1$, we deduce that 
$$g(s)\leq \left(\frac{1}{10}\right)^s + (0.071797)^s \quad \textrm{and} \quad h(s)\leq (0.012197)^s + (0.016134)^s +\left(\frac{1}{84}\right)^s$$
Therefore, $\max\{g(0.281266), h(0.281266)\} < 0.999999$ and, \emph{a fortiori}, the $(0.281266)$-Hausdorff measure of 
$$K=\{[a_0;a_1,\dots]: \sqrt{13} < m(\underline{a}) < 3.84\}$$ 
is finite. It follows from this discussion: 
\begin{proposition}\label{p.gap-Cantor-sqrt13-I}$(M\setminus L)\cap(\sqrt{13}, 3.84)\subset X_3(\{13,31\})+K$ where 
$$X_3(\{13,31\}):=\{[0;\gamma]:\gamma\in \{1,2,3\}^{\mathbb{N}} \textrm{ doesn't contain } 13 \textrm{ nor } 31\}$$ and $K$ is a set of Hausdorff dimension $HD(K)<0.281266$.
\end{proposition} 

A direct consequence of this proposition is:

\begin{corollary}\label{c.M-L-sqrt13-I} $HD((M\setminus L)\cap(\sqrt{13}, 3.84)) < 0.986927$.
\end{corollary}

\begin{proof} By Proposition \ref{p.gap-Cantor-sqrt13-I} and Hensley's estimate \cite{He} for $HD(K(\{1,2,3\}))$, one has 
\begin{eqnarray*}
HD((M\setminus L)\cap(\sqrt{13},3.84)) &\leq& HD(K(\{1,2,3\}))+HD(K) \\ 
&<& 0.705661 + 0.281266 = 0.986927
\end{eqnarray*}
This completes the argument. 
\end{proof}

\subsection{Description of $(M\setminus L)\cap(3.84, \sqrt{20})$}\label{ss.3.84-sqrt20} 

Let $m\in M\setminus L$ with $3.84<m<\sqrt{20}$. In this setting, $m=m(\underline{a})=f(\underline{a})$ for a sequence $\underline{a}=(\dots, a_{-1}, a_0, a_1, \dots)\in \{1,2,3\}^{\mathbb{Z}}=:\Sigma(C)$. Consider the subshift $\Sigma(B)\subset\Sigma(C)$ associated to $B=\{1,2, 2321, 1232\}$. Note that $B$ and $C$ are symmetric, and the quantity $c(B,C)$ introduced above is bounded by 
\begin{equation*}
c(B,C) \leq [3;2,1,\overline{1,2}]+[0;2,\overline{3,1}] < 3.83 < 3.84 < m
\end{equation*} 
thanks to Lemma \ref{l.0}. 

As it was explained before, we can use Lemma \ref{l.key} to see that, up to transposing $\underline{a}$, there exists $k\in\mathbb{N}$ such that, for all $n\geq k$, either $\dots a_0^*\dots a_n$ has a forced continuation $\dots a_0^*\dots a_n a_{n+1}\dots$ or two continuations $\dots a_0^*\dots a_n\alpha_n$ and $\dots a_0^*\dots a_n\beta_n$ with $[[0;\alpha_n],[0;\beta_n]]\cap K(B)=\emptyset$. From this, we are ready to set up an efficient cover of $(M\setminus L)\cap(3.84,\sqrt{20})$. In this direction, note that the condition $[[0;\alpha_n],[0;\beta_n]]\cap K(B)=\emptyset$ imposes two types of restrictions:
\begin{itemize}
\item $\alpha_n=3\alpha_{n+1}$ and $\beta_n=21\beta_{n+2}$;
\item $\alpha_n=23\alpha_{n+2}$ and $\beta_n\in\{1121\beta_{n+4},113\beta_{n+3}\}$.
\end{itemize} 

Hence, we have that the $s$-Hausdorff measure of the set 
$$K:=\{\underline{a}=[a_0;a_1,\dots]: 3.84 < m(\underline{a}) < \sqrt{20}\}$$ is finite for any parameter $s$ with 
$$g(s)=\frac{|I(a_1,\dots,a_n,3)|^s + |I(a_1,\dots,a_n,2,1)|^s}{|I(a_1,\dots,a_n)|^s}<1$$
and 
$$h(s)=\frac{|I(a_1,\dots,a_n,2,3)|^s+|I(a_1,\dots,a_n,1,1,2,1)|^s+|I(a_1,\dots,a_n,1,1,3)|^s}{|I(a_1,\dots,a_n)|^s}<1$$ 
for all $(a_1,\dots,a_n)\in\bigcup\limits_{k\in\mathbb{N}}\{1,2,3\}^k$. 

We saw in the previous subsection that $g(0.281266) < 0.999999$ and 
$$h(s)\leq (0.016134)^s +\left(\frac{1}{84}\right)^s + \frac{|I(a_1,\dots,a_n,1,1,3)|^s}{|I(a_1,\dots,a_n)|^s}$$ 
Because the recurrence $q_{j+2}=a_{j+2}q_{j+1}+q_j$ implies that 
$$\frac{|I(a_1,\dots,a_n,1,1,3)|}{|I(a_1,\dots,a_n)|} = \frac{r+1}{(4r+7)(5r+9)}$$
where $r=q_{n-1}/q_n\in(0,1)$, and since $\frac{r+1}{(4r+7)(5r+9)}\leq \frac{1}{63}$ for all $0\leq r\leq 1$, we conclude 
$$h(s)\leq (0.016134)^s +\left(\frac{1}{84}\right)^s + \left(\frac{1}{63}\right)^s$$ 

Thus, $\max\{g(0.281266), h(0.281266)\} < 0.999999$ and, \emph{a fortiori}, the $(0.281266)$-Hausdorff measure of 
$$K=\{[a_0;a_1,\dots]: 3.84 < m(\underline{a}) < \sqrt{20}\}$$ 
is finite. In particular, we proved the following result: 
\begin{proposition}\label{p.gap-Cantor-sqrt13-II}$(M\setminus L)\cap(3.84, \sqrt{20})\subset K(\{1,2,3\})+K$ where $K$ is a set of Hausdorff dimension $HD(K)<0.281266$.
\end{proposition} 

As usual, this proposition yields the following estimate: 

\begin{corollary}\label{c.M-L-sqrt13-II} $HD((M\setminus L)\cap(3.84, \sqrt{20})) < 0.986927$.
\end{corollary}

\begin{proof} By Proposition \ref{p.gap-Cantor-sqrt13-II} and Hensley's estimate \cite{He} for $HD(K(\{1,2,3\}))$, one has 
\begin{eqnarray*}
HD((M\setminus L)\cap(3.84,\sqrt{20})) &\leq& HD(K(\{1,2,3\}))+HD(K) \\ 
&<& 0.705661 + 0.281266 = 0.986927
\end{eqnarray*}
This ends the proof. 
\end{proof}

\subsection{Description of $(M\setminus L)\cap(\sqrt{20}, \sqrt{21})$} 

Let $m\in M\setminus L$ with $\sqrt{20}<m<\sqrt{21}$. In this setting, $m=m(\underline{a})=f(\underline{a})$ for a sequence 
$$\underline{a}=(\dots, a_{-1}, a_0, a_1, \dots)\in \Sigma(C):=\{\gamma\in \{1,2,3,4\}^{\mathbb{Z}} \textrm{ not containing } 14, 41, 24, 42\}$$ because:
\begin{itemize} 
\item if $\underline{\theta}$ contains $14$ or $41$, then $m(\underline{\theta})\geq [4;1,\overline{1,4}]+[0;\overline{4,1}]>\sqrt{21}$ by Lemma \ref{l.0};
\item if $\underline{\theta}$ contains $24$ or $42$ but neither $14$ nor $41$, then Lemma \ref{l.0} implies that $m(\underline{\theta})\geq [4;2,\overline{1,3}]+[0;\overline{4,2}]>\sqrt{21}$.  
\end{itemize} 

Consider the subshift $\Sigma(B)\subset\Sigma(C)$ associated to $B=\{21312, 232, 3, 11313, 31311\}$ with the restrictions that $31311$ follows only $3$, and $11313$ is followed only by $3$. Note that $B$ and $C$ are symmetric, and the quantity $c(B,C)$ introduced above is  
\begin{equation*}
c(B,C) < 4.46 <\sqrt{20} < m
\end{equation*} 
thanks to Lemma \ref{l.0}. 

As it was explained before, we can use Lemma \ref{l.key} to see that, up to transposing $\underline{a}$, there exists $k\in\mathbb{N}$ such that, for all $n\geq k$, either $\dots a_0^*\dots a_n$ has a forced continuation $\dots a_0^*\dots a_n a_{n+1}\dots$ or two continuations $\dots a_0^*\dots a_n\alpha_n$ and $\dots a_0^*\dots a_n\beta_n$ with $[[0;\alpha_n],[0;\beta_n]]\cap K(B)=\emptyset$. From this, we are ready to cover $(M\setminus L)\cap(\sqrt{20},\sqrt{21})$. In this direction, note that the condition $[[0;\alpha_n],[0;\beta_n]]\cap K(B)=\emptyset$ imposes three types of restrictions:
\begin{itemize}
\item $\alpha_n=4\alpha_{n+1}$ and $\beta_n=3131\beta_{n+4}$;
\item $\alpha_n\in\{33131\alpha_{n+5}, 34\alpha_{n+2}\}$ and $\beta_n=2131\beta_{n+4}$;
\item $\alpha_n=23\alpha_{n+2}$ and $\beta_n=1131\beta_{n+3}$.
\end{itemize} 

Hence, we have that the $s$-Hausdorff measure of the set 
$$K:=\{[a_0;a_1,\dots]: \sqrt{20} < m(\underline{a}) < \sqrt{21}\}$$ is finite for any parameter $s$ with 
$$g(s)=\frac{|I(a_1,\dots,a_n,4)|^s + |I(a_1,\dots,a_n,3,1,3,1)|^s}{|I(a_1,\dots,a_n)|^s}<1,$$
$$h(s)=\frac{|I(a_1,\dots,a_n,3,3,1,3,1)|^s+|I(a_1,\dots,a_n,3,4)|^s+|I(a_1,\dots,a_n,2,1,3,1)|^s}{|I(a_1,\dots,a_n)|^s}<1$$ 
and 
$$i(s)=\frac{|I(a_1,\dots,a_n,2,3)|^s+|I(a_1,\dots,a_n,1,1,3,1)|^s}{|I(a_1,\dots,a_n)|^s}<1$$
for all $(a_1,\dots,a_n)\in\bigcup\limits_{k\in\mathbb{N}}\{1,2,3,4\}^k$. 

We saw in the previous subsection that  
$$i(s)\leq (0.016134)^s +\frac{|I(a_1,\dots,a_n,1,1,3,1)|^s}{|I(a_1,\dots,a_n)|^s}$$ 
On the other hand, the recurrence $q_{j+2}=a_{j+2}q_{j+1}+q_j$ implies that 
$$\frac{|I(a_1,\dots,a_n,4)|}{|I(a_1,\dots,a_n)|} = \frac{r+1}{(r+4)(r+5)}, \quad \frac{|I(a_1,\dots,a_n,3,1,3,1)|}{|I(a_1,\dots,a_n)|} = \frac{r+1}{(5r+19)(9r+24)},$$
\begin{eqnarray*}
&\frac{|I(a_1,\dots,a_n,3,3,1,3,1)|}{|I(a_1,\dots,a_n)|} = \frac{r+1}{(19r+62)(34r+111)}, \quad \frac{|I(a_1,\dots,a_n,3,4)|}{|I(a_1,\dots,a_n)|} = \frac{r+1}{(4r+13)(5r+16)} \\ 
&\frac{|I(a_1,\dots,a_n,2,1,3,1)|}{|I(a_1,\dots,a_n)|} =  \frac{r+1}{(5r+14)(9r+25)}
\end{eqnarray*}
and 
$$\frac{|I(a_1,\dots,a_n,1,1,3,1)|}{|I(a_1,\dots,a_n)|} = \frac{r+1}{(5r+9)(9r+16)}$$
where $r=q_{n-1}/q_n\in(0,1)$. 

Since $\frac{r+1}{(r+4)(r+5)}\leq \frac{1}{15}$, $\frac{r+1}{(5r+19)(9r+24)}\leq \frac{1}{516}$, $\frac{r+1}{(19r+62)(34r+111)} \leq \frac{2}{11745}$, $\frac{r+1}{(4r+13)(5r+16)} \leq \frac{2}{357}$, $\frac{r+1}{(5r+14)(9r+25)}\leq 0.003106$ and $\frac{r+1}{(5r+9)(9r+16)} \leq \frac{1}{144}$ for all $0\leq r\leq 1$, we get  
\begin{eqnarray*}
& & g(s)\leq \left(\frac{1}{15}\right)^s + \left(\frac{1}{516}\right)^s, \quad   h(s)\leq \left(\frac{2}{11745}\right)^s + \left(\frac{2}{357}\right)^s + (0.003106)^s, \\
& & i(s)\leq (0.016134)^s + \left(\frac{1}{144}\right)^s
\end{eqnarray*} 

Thus, $\max\{g(0.172825), h(0.172825), i(0.172825)\} < 0.999997$ and, \emph{a fortiori}, the $(0.172825)$-Hausdorff measure of 
$$K=\{[a_0;a_1,\dots]: \sqrt{20} < m(\underline{a}) < \sqrt{21}\}$$ 
is finite. In particular, we proved the following result: 
\begin{proposition}\label{p.gap-Cantor-sqrt20}$(M\setminus L)\cap(\sqrt{20}, \sqrt{21})\subset X_4(\{14,41,24,42\})+K$ where 
$$X_4(\{14,41,24,42\}):=\{[0;\gamma]:\gamma\in \{1,2,3,4\}^{\mathbb{N}} \textrm{ doesn't contain } 14,  41, 24, 42\}$$ 
and $K$ is a set of Hausdorff dimension $HD(K)<0.172825$.
\end{proposition} 

As usual, this proposition yields the following estimate: 

\begin{corollary}\label{c.M-L-sqrt20} $HD((M\setminus L)\cap(\sqrt{20}, \sqrt{21})) < 0.961772$.
\end{corollary}

\begin{proof} By Proposition \ref{p.gap-Cantor-sqrt20} and Hensley's estimate \cite{He} for $HD(K(\{1,2,3,4\}))$, one has 
\begin{eqnarray*}
HD((M\setminus L)\cap(\sqrt{20},\sqrt{21})) &\leq& HD(K(\{1,2,3,4\}))+HD(K) \\ 
&<& 0.788947 + 0.172825 = 0.961772 
\end{eqnarray*}
This finishes the argument. 
\end{proof}

\subsection{End of proof of Theorem \ref{t.prrova}}

By Corollaries \ref{c.M-L-sqrt10-sqrt13}, \ref{c.M-L-sqrt13-I}, \ref{c.M-L-sqrt13-II}, \ref{c.M-L-sqrt20}, we have that 
$$HD((M\setminus L)\cap (\sqrt{10}, \sqrt{21})) < 0.986927$$

On the other hand, Freiman \cite{Fr73b} and Schecker \cite{Sch77} proved that $[\sqrt{21},\infty)\subset L$. Therefore, $(M\setminus L)\cap [\sqrt{21},\infty)=\emptyset$. 

Also, the proof of Theorem 1 in Chapter 6 of Cusick--Flahive's book \cite{CF} gives the (rigorous) estimate 
$$HD(M\cap(-\infty,\sqrt{10})) < 0.93$$
for the Hausdorff dimension of $M\cap(-\infty,\sqrt{10})$. 

By putting these facts together, we obtain the desired conclusion, namely, 
$$HD(M\setminus L) < 0.986927$$

\appendix

\section{Structure of $M\setminus L$ near $3.7$}\label{s.Cantor-Cusick}

\subsection{The largest interval $J$ avoiding $L$ containing $C$} 

Consider the quantities 
$$j_0:=\lambda_0(\overline{33^*22212}) = 3.70969985967967\dots$$
and 
$$j_1:=\lambda_0(\overline{21}12212332221233^*22212332221233321\overline{12}) = 3.70969985975042\dots$$

By Corollary \ref{c.7}, Lemma \ref{l.8}, and the proof of Corollary \ref{c.replication}, we have that 
\begin{proposition}\label{p.almost-replication} If $j_0 < m(a) = \lambda_0(a) < 3.7096998599$, then  (up to transposition) 
\begin{itemize}
\item either $a=\dots12212332221233^*222123322212\dots$
\item or $a=\dots 23322212332221233^*222123322212\dots$
and the vicinity of the position $-7$ is $\dots2332221233^*222123322\dots$
\end{itemize} 
\end{proposition} 

Indeed, this happens because $m(a)<3.7096998599$ allows to use all results from Sections \ref{s.uniqueness} and \ref{s.replication} \emph{except} for Lemma \ref{l.replication}. 

\begin{proposition}\label{p.j1} If $m(a)<3.71$ and $a$ contains $12212332221233^*222123322212$, then $m(a)\geq j_1$. 
\end{proposition}

\begin{proof} By Lemma \ref{l.1} (i), (ii), 
\begin{eqnarray*}
\lambda_0(a)&=&[3;2,2,2,1,2,3,3,2,2,2,1,2,\dots]+[0;3,2,1,2,2,2,3,3,2,1,2,2,1,\dots] \\ 
&\geq& [3;2,2,2,1,2,3,3,2,2,2,1,2,3,3,3,2,\dots] + [0;3,2,1,2,2,2,3,3,2,1,2,2,1,\overline{1,2}] 
\end{eqnarray*}
By Lemma \ref{l.2} (v) and Lemma \ref{l.3} (vi), 
\begin{eqnarray*}
\lambda_0(a)&\geq&[3;2,2,2,1,2,3,3,2,2,2,1,2,3,3,3,2,\dots] + [0;3,2,1,2,2,2,3,3,2,1,2,2,1,\overline{1,2}]  \\ 
&\geq& [3;2,2,2,1,2,3,3,2,2,2,1,2,3,3,3,2,1,\dots] + [0;3,2,1,2,2,2,3,3,2,1,2,2,1,\overline{1,2}] 
\end{eqnarray*}
By Lemma \ref{l.1} (i), we conclude that 
\begin{eqnarray*}
\lambda_0(a)&\geq& [3;2,2,2,1,2,3,3,2,2,2,1,2,3,3,3,2,1,\dots] + [0;3,2,1,2,2,2,3,3,2,1,2,2,1,\overline{1,2}] \\ 
&\geq& [3;2,2,2,1,2,3,3,2,2,2,1,2,3,3,3,2,1,\overline{1,2}] + [0;3,2,1,2,2,2,3,3,2,1,2,2,1,\overline{1,2}] \\ 
&=& j_1 
\end{eqnarray*}
\end{proof}

By putting together Propositions \ref{p.almost-replication} and \ref{p.j1}, we obtain the following strengthening of Propositions \ref{p.1} and \ref{p.2}:

\begin{proposition}\label{p.1'} The open interval $J=(j_0, j_1)$ containing $C$ is disjoint from $L$. 
\end{proposition} 

\begin{proof} On one hand, the fact that $J$ contains $C$ is an immediate consequence of Proposition \ref{p.2}. On the other hand, if $j_0<m(a)<j_1$ for a periodic sequence $a$, then, thanks to Proposition \ref{p.j1}, we would be able to iteratively apply Proposition \ref{p.almost-replication} to obtain that $m(a)=\lambda_0(\overline{33^*22212})=j_0$, a contradiction. Since the Lagrange spectrum is the closure of Markov values associated to periodic sequences, we derive that $J\cap L=\emptyset$. 
\end{proof}

Since it is not hard to see that $j_0$ and $j_1$ belong to $L$, the previous proposition implies that: 

\begin{corollary} $J$ is the largest interval containing $C$ which is disjoint from $L$. 
\end{corollary}

\subsection{The largest known element of $M\setminus L$}

Consider the quantity 

$$\Upsilon := \lambda_0(\overline{3322212}33^*222123322212212121\overline{12}) = 3.7096998597503806\dots$$

\begin{proposition} $\Upsilon$ is the largest element of $(M\setminus L)\cap J$. 
\end{proposition}

\begin{proof} Given $a\in\{1,2,3\}^{\mathbb{Z}}$ with $m(a)=\lambda_0(a)\in J$, we can apply Propositions \ref{p.j1} and \ref{p.almost-replication} to obtain that (up to transposition) 
$$m(a)=[3;2,2,2,1,2,3,3,2,2,2,1,2,\dots]+[0;3,\overline{2,1,2,2,2,3,3}]$$ 

If $a_{13}=1$, then $m(a)\geq [3;2,2,2,1,2,3,3,2,2,2,1,2,1,\overline{1,2}]+[0;3,\overline{2,1,2,2,2,3,3}] = 3.7096998599\dots$, a contradiction. Hence, 
$$m(a)\leq [3;2,2,2,1,2,3,3,2,2,2,1,2,2,\dots]+[0;3,\overline{2,1,2,2,2,3,3}]$$

If $a_{14}\in\{2,3\}$, then $m(a)\geq [3;2,2,2,1,2,3,3,2,2,2,1,2,2,2,3,3,\overline{2,1}]+[0;3,\overline{2,1,2,2,2,3,3}] = 3.709699859799\dots$, a contradiction. Thus, 
$$m(a)\leq [3;2,2,2,1,2,3,3,2,2,2,1,2,2,1\dots]+[0;3,\overline{2,1,2,2,2,3,3}]$$

If $a_{15}=1$, then $m(a)\geq [3;2,2,2,1,2,3,3,2,2,2,1,2,2,1,1,\overline{1,2}]+[0;3,\overline{2,1,2,2,2,3,3}] = 3.709699859765\dots$, a contradiction. Therefore, 
$$m(a)\leq [3;2,2,2,1,2,3,3,2,2,2,1,2,2,1,2,\dots]+[0;3,\overline{2,1,2,2,2,3,3}]$$

If $a_{16}\in\{2,3\}$, then $m(a)\geq [3;2,2,2,1,2,3,3,2,2,2,1,2,2,1,2,2,\overline{3,1}\dots] + [0;3,\overline{2,1,2,2,2,3,3}] = 3.709699859753\dots$, a contradiction. So, 
$$m(a)\leq [3;2,2,2,1,2,3,3,2,2,2,1,2,2,1,2,1,\dots]+[0;3,\overline{2,1,2,2,2,3,3}]$$

If $a_{17}=1$, or $a_{17}=2$ and $a_{18}\in\{2,3\}$, then $m(a)\geq [3;2,2,2,1,2,3,3,2,2,2,1,2,2,1,2,1,2,2,\overline{3,1}]+[0;3,\overline{2,1,2,2,2,3,3}] = 3.70969985975049\dots$, a contradiction. Hence, 
$$m(a)\leq [3;2,2,2,1,2,3,3,2,2,2,1,2,2,1,2,1,2,1,\dots]+[0;3,\overline{2,1,2,2,2,3,3}]$$

It follows that 
$$m(a)\leq [3;2,2,2,1,2,3,3,2,2,2,1,2,2,1,2,1,2,1,\overline{1,2}]+[0;3,\overline{2,1,2,2,2,3,3}] = \Upsilon$$
This completes the argument. 
\end{proof}

\subsection{The Hausdorff dimension of $M\setminus L$ near $3.7$}

As it is explained in our previous works \cite{MaMo1} and \cite{MaMo2}, we have that 
$$HD((M\setminus L)\cap J) = HD(\Omega)$$ 
where $\Omega$ is the Gauss--Cantor set 
$$\Omega:=\{[0;\gamma]:\gamma\in\{1,2,3\}^{\mathbb{N}} \textrm{ doesn't contain subwords in } P \}$$ with $P$ consisting of ``big words''\footnote{In the sense that the appearance of these words implies that the value of $\lambda_0$ surpasses $j_1$.} appearing in items (i), (ii), (v), (vi), (vii), (x), (xii), (xiii), (xiv), (xv), (xvii), (xxii), (xxv), (xxvi), (xxviii), (xxix), (xxx), (xxxi), (xxxiii), (xxxv), (xxxvi), (xxxvii), (xxxviii), (xxxix) in Section \ref{s.uniqueness} and \ref{s.replication} and their transposes, and the ``self-replicating'' word $2332221233222123322$ in Corollary \ref{c.replication} and its transpose. 

\section{Empirical derivation of $HD(M\setminus L)<0.888$}\label{s.prova}

The algorithm developed by Jenkinson--Pollicott in \cite{JP01} allows to give \emph{heuristic} estimates for the Hausdorff dimensions of certain Cantor sets of real numbers whose continued fraction expansions satisfy some constraints. Furthermore, Jenkinson--Pollicott shows in \cite{JP16} how these \emph{empirical} estimates can be converted in \emph{rigorous} estimates. 

In this section, we will explore Jenkinson--Pollicott algorithm to give an \emph{empirical} derivation of the following bound: 

\begin{equation}\label{e.prova} 
HD(M\setminus L) < 0.888
\end{equation}

\subsection{Heuristic estimates for $HD((M\setminus L)\cap (-\infty, \sqrt{13}))$} 

Consider again the subshift $\Sigma(B):=\{11, 22\}^{\mathbb{Z}}$ of $\Sigma(C)=\{1,2\}^{\mathbb{Z}}$. The quantity $c(B,C)$ introduced above is 
$$c(B,C) = [2;\overline{1,1}] + [0;2,\overline{2,1}] < 3.0407<3.06$$

This refined information on $c(B,C)$ allows us to improve Proposition \ref{p.gap-Cantor-sqrt12} and Corollary \ref{c.M-L-sqrt10-sqrt13}. Indeed, by repeating the analysis of Subsection \ref{ss.basic-argument} with this stronger estimate on $c(B,C)$, one gets the following result:

\begin{proposition} $(M\setminus L)\cap(3.06, \sqrt{13})\subset K(\{1,2\})+K$ where $K$ is a set of Hausdorff dimension $HD(K)<0.174813$. In particular, $HD((M\setminus L)\cap(3.06,\sqrt{13})) < 0.706104$.
\end{proposition} 

This proposition implies that 
$$HD((M\setminus L)\cap(-\infty,\sqrt{13}))<\max\{HD((M\setminus L)\cap(-\infty,3.06)), 0.706104\}$$

On the other hand, as it is explained in Table 1 of Chapter 5 of Cusick--Flahive's book \cite{CF}, a result due to Jackson implies that if the Markov value of a sequence $\underline{a}\in\Sigma$ is $m(\underline{a})<3.06$, then $\underline{a}$ doesn't contain $1,2,1$ nor $2,1,2$. Thus, 
$$HD((M\setminus L)\cap(-\infty,3.06))\leq 2 \cdot HD(K(X_2(\{121,212\})))$$
where $K(X_2(\{121,212\})) = \{[0;\gamma]: \gamma\in\{1,2\}^{\mathbb{N}} \textrm{ not containing } 121, 212\}$. 

A quick implementation of the Jenkinson--Pollicott algorithm seems to indicate that 
$$HD(K(X_2(\{121,212\})))<0.365$$ 

Hence, our discussion so far gives that 
\begin{equation}\label{e.prova-i} 
HD((M\setminus L)\cap (-\infty,\sqrt{13})) < 0.73
\end{equation}

\subsection{Heuristic estimates for $HD((M\setminus L)\cap(\sqrt{13}, 3.84))$}

Our Proposition \ref{p.gap-Cantor-sqrt13-I} above implies that 
$$HD((M\setminus L)\cap(\sqrt{13}, 3.84)) < HD(X_3(\{13, 31\}))+0.281266$$
where $X_3(\{13,31\}):=\{[0;\gamma]:\gamma\in \{1,2,3\}^{\mathbb{N}} \textrm{ doesn't contain } 13 \textrm{ nor } 31\}$. 

After running Jenkinson--Pollicott algorithm, one seems to get that 
$$HD(X_3(\{13, 31\}))<0.574$$
and, \emph{a fortiori},  
\begin{equation}\label{e.prova-ii}
HD((M\setminus L)\cap(\sqrt{13}, 3.84)) < 0.856
\end{equation}

\subsection{Heuristic estimates for $HD((M\setminus L)\cap(3.84, 3.92))$}

Let $m\in M\setminus L$ with $3.84<m<3.92$. In this setting, $m=m(\underline{a})=f(\underline{a})$ for a sequence $\underline{a}=(\dots, a_{-1}, a_0, a_1, \dots)\in \{1,2,3\}^{\mathbb{Z}}=:\Sigma(C)$ not containing $131$, $313$, $231$, $132$. Consider the subshift $\Sigma(B)\subset\Sigma(C)$ associated to $B=\{1,2, 2321, 1232, 33\}$ with the restrictions that $33$ doesn't follow $1$ or $2321$, and $33$ is not followed by $1$ or $2321$. Note that $B$ and $C$ are symmetric, and the quantity $c(B,C)$ introduced above is 
\begin{equation*}
c(B,C) < 3.84 < m
\end{equation*} 
thanks to Lemma \ref{l.0}. 

From Lemma \ref{l.key} we obtain that, up to transposing $\underline{a}$, there exists $k\in\mathbb{N}$ such that, for all $n\geq k$, either $\dots a_0^*\dots a_n$ has a forced continuation $\dots a_0^*\dots a_n a_{n+1}\dots$ or two continuations $\dots a_0^*\dots a_n\alpha_n$ and $\dots a_0^*\dots a_n\beta_n$ with $[[0;\alpha_n],[0;\beta_n]]\cap K(B)=\emptyset$. From this, we are ready to set up an efficient cover of $(M\setminus L)\cap(3.84,3.92)$. In this direction, note that the condition $[[0;\alpha_n],[0;\beta_n]]\cap K(B)=\emptyset$ imposes two types of restrictions:
\begin{itemize}
\item $\alpha_n=33\alpha_{n+3}$ and $\beta_n=21\beta_{n+3}$;
\item $\alpha_n=23\alpha_{n+2}$ and $\beta_n\in\{113\beta_{n+3}, 1121\beta_{n+4}\}$.
\end{itemize} 

Hence, we have that the $s$-Hausdorff measure of the set 
$$K:=\{[a_0;a_1,\dots]: 3.84 < m(\underline{a}) < 3.92\}$$ is finite for any parameter $s$ with 
$$g(s)=\frac{|I(a_1,\dots,a_n,3,3)|^s + |I(a_1,\dots,a_n,2,1)|^s}{|I(a_1,\dots,a_n)|^s}<1$$
and 
$$h(s)=\frac{|I(a_1,\dots,a_n,2,3)|^s+|I(a_1,\dots,a_n,1,1,3)|^s+|I(a_1,\dots,a_n,1,1,2,1)|^s}{|I(a_1,\dots,a_n)|^s}<1$$ 
for all $(a_1,\dots,a_n)\in\bigcup\limits_{k\in\mathbb{N}}\{1,2,3\}^k$. 

We saw in Subsections \ref{ss.sqrt13-3.84} and \ref{ss.3.84-sqrt20} that 
$$g(s)\leq \frac{|I(a_1,\dots,a_n,3,3)|^s}{|I(a_1,\dots,a_n)|^s} + (0.071797)^s \quad \textrm{and} \quad h(s)\leq (0.016134)^s + \left(\frac{1}{63}\right)^s+\left(\frac{1}{84}\right)^s$$ 
Because the recurrence $q_{j+2}=a_{j+2}q_{j+1}+q_j$ implies that 
\begin{equation*}
\frac{|I(a_1,\dots,a_n,3,3)|}{|I(a_1,\dots,a_n)|} = \frac{r+1}{(3r+10)(4r+13)} 
\end{equation*}
where $r=q_{n-1}/q_n\in(0,1)$, and since $\frac{r+1}{(3r+10)(4r+13)}  \leq \frac{2}{221}$ for all $0\leq r\leq 1$, we get 
$$g(s)\leq \left(\frac{2}{221}\right)^s + (0.071797)^s$$ 

Thus, $\max\{g(0.25966), h(0.25966)\} < 0.99999$ and, \emph{a fortiori}, the $(0.25966)$-Hausdorff measure of 
$$K=\{[a_0;a_1,\dots]: 3.84 < m(\underline{a}) < 3.92\}$$ 
is finite. In particular, we proved the following result: 
\begin{proposition}\label{p.gap-Cantor-sqrt13-II}$(M\setminus L)\cap(3.84, 3.92)\subset X_3(\{131, 313, 231, 132\})+K$ where 
$$X_3(\{131, 313, 231, 132\}):=\{[0;\gamma]:\gamma\in \{1,2,3\}^{\mathbb{N}} \textrm{ not containing } 131, 313, 231, 132\}$$ and $K$ is a set of Hausdorff dimension $HD(K)<0.25966$.
\end{proposition} 

After running Jenkinson--Pollicott algorithm, one seems to obtain that 
$$HD(X_3(\{131, 313, 231, 132\}))<0.612$$
so that the previous proposition indicates that   
\begin{equation}\label{e.prova-iii}
HD((M\setminus L)\cap(3.84, 3.92)) < 0.872
\end{equation}

\subsection{Heuristic estimates for $HD((M\setminus L)\cap(3.92, 4.01))$}\label{ss.3.9-4.01}

Let $m\in M\setminus L$ with $3.92<m<4.01$. In this setting, $m=m(\underline{a})=f(\underline{a})$ for a sequence $\underline{a}=(\dots, a_{-1}, a_0, a_1, \dots)\in \{1,2,3\}^{\mathbb{Z}}=:\Sigma(C)$ not containing $131$, $313$, $2312$, $2132$. Consider the subshift $\Sigma(B)\subset\Sigma(C)$ associated to $B=\{1,2, 211, 112, 232, 1133, 3311\}$ with the restrictions that $3311$ comes only after $211$ and $3311$ has to follow $2$, and $1133$ has to appear after $2$, and $1133$ has to follow $112$. Note that $B$ and $C$ are symmetric, and the quantity $c(B,C)$ introduced above is 
\begin{equation*}
c(B,C) < 3.92 < m
\end{equation*} 
thanks to Lemma \ref{l.0}. 

From Lemma \ref{l.key} we obtain that, up to transposing $\underline{a}$, there exists $k\in\mathbb{N}$ such that, for all $n\geq k$, either $\dots a_0^*\dots a_n$ has a forced continuation $\dots a_0^*\dots a_n a_{n+1}\dots$ or two continuations $\dots a_0^*\dots a_n\alpha_n$ and $\dots a_0^*\dots a_n\beta_n$ with $[[0;\alpha_n],[0;\beta_n]]\cap K(B)=\emptyset$. From this, we are ready to set up an efficient cover of $(M\setminus L)\cap(3.92,4.01)$. In this direction, note that the condition $[[0;\alpha_n],[0;\beta_n]]\cap K(B)=\emptyset$ imposes two types of restrictions:
\begin{itemize}
\item $\alpha_n=331\alpha_{n+3}$ and $\beta_n=21\beta_{n+3}$;
\item $\alpha_n=23\alpha_{n+2}$ and $\beta_n=113\beta_{n+3}$.
\end{itemize} 

Hence, we have that the $s$-Hausdorff measure of the set 
$$K:=\{[a_0;a_1,\dots]: 3.92 < m(\underline{a}) < 4.01\}$$ is finite for any parameter $s$ with 
$$g(s)=\frac{|I(a_1,\dots,a_n,3,3,1)|^s + |I(a_1,\dots,a_n,2,1)|^s}{|I(a_1,\dots,a_n)|^s}<1$$
and 
$$h(s)=\frac{|I(a_1,\dots,a_n,2,3)|^s+|I(a_1,\dots,a_n,1,1,3)|^s}{|I(a_1,\dots,a_n)|^s}<1$$ 
for all $(a_1,\dots,a_n)\in\bigcup\limits_{k\in\mathbb{N}}\{1,2,3\}^k$. 

We saw in Subsections \ref{ss.sqrt13-3.84} and \ref{ss.3.84-sqrt20} that 
$$g(s)\leq \frac{|I(a_1,\dots,a_n,3,3,1)|^s}{|I(a_1,\dots,a_n)|^s} + (0.071797)^s \quad \textrm{and} \quad h(s)\leq (0.016134)^s + \left(\frac{1}{63}\right)^s$$ 
Because the recurrence $q_{j+2}=a_{j+2}q_{j+1}+q_j$ implies that 
\begin{equation*}
\frac{|I(a_1,\dots,a_n,3,3,1)|}{|I(a_1,\dots,a_n)|} = \frac{r+1}{(4r+13)(7r+23)} 
\end{equation*}
where $r=q_{n-1}/q_n\in(0,1)$, and since $\frac{r+1}{(4r+13)(7r+23)} \leq \frac{1}{255}$ for all $0\leq r\leq 1$, we get 
$$g(s)\leq \left(\frac{1}{255}\right)^s + (0.071797)^s$$ 

Thus, $\max\{g(0.177645), h(0.177645)\} < 0.99999$ and, \emph{a fortiori}, the $(0.177645)$-Hausdorff measure of 
$$K=\{[a_0;a_1,\dots]: 3.92 < m(\underline{a}) < 4.01\}$$ 
is finite. In particular, we proved the following result: 
\begin{proposition}\label{p.gap-Cantor-sqrt13-II}$(M\setminus L)\cap(3.92, 4.01)\subset X_3(\{131, 313, 2312, 2132\})+K$ where 
$$X_3(\{131, 313, 2312, 2132\}):=\{[0;\gamma]:\gamma\in \{1,2,3\}^{\mathbb{N}} \textrm{ not containing } 131, 313, 2312, 2132\}$$ and $K$ is a set of Hausdorff dimension $HD(K)<0.167655$.
\end{proposition} 

After running Jenkinson--Pollicott algorithm, one seems to obtain that 
$$HD(X_3(\{131, 313, 2312, 2132\}))<0.65$$
so that the previous proposition indicates that   
\begin{equation}\label{e.prova-iv}
HD((M\setminus L)\cap(3.92, 4.01)) < 0.828
\end{equation}

\subsection{Heuristic estimates for $HD((M\setminus L)\cap(4.01,\sqrt{20}))$} 

Let $m\in M\setminus L$ with $4.01<m<\sqrt{20}$. In this setting, $m=m(\underline{a})=f(\underline{a})$ for a sequence $\underline{a}=(\dots, a_{-1}, a_0, a_1, \dots)\in \{1,2,3\}^{\mathbb{Z}}=:\Sigma(C)$. Consider the subshift $\Sigma(B)\subset\Sigma(C)$ associated to 
$$B=\{11,2, 232, 213312, 33\}$$ Note that $B$ and $C$ are symmetric, and the quantity $c(B,C)$ introduced above is 
\begin{equation*}
c(B,C) < 4.01 < m
\end{equation*} 
thanks to Lemma \ref{l.0}. 

From Lemma \ref{l.key} we obtain that, up to transposing $\underline{a}$, there exists $k\in\mathbb{N}$ such that, for all $n\geq k$, either $\dots a_0^*\dots a_n$ has a forced continuation $\dots a_0^*\dots a_n a_{n+1}\dots$ or two continuations $\dots a_0^*\dots a_n\alpha_n$ and $\dots a_0^*\dots a_n\beta_n$ with $[[0;\alpha_n],[0;\beta_n]]\cap K(B)=\emptyset$. From this, we are ready to set up an efficient cover of $(M\setminus L)\cap(4.01,\sqrt{20})$. In this direction, note that the condition $[[0;\alpha_n],[0;\beta_n]]\cap K(B)=\emptyset$ imposes two types of restrictions:
\begin{itemize}
\item $\alpha_n=331\alpha_{n+3}$ and $\beta_n=213\beta_{n+3}$;
\item $\alpha_n=23\alpha_{n+2}$ and $\beta_n=113\beta_{n+3}$.
\end{itemize} 

Hence, we have that the $s$-Hausdorff measure of the set 
$$K:=\{[a_0;a_1,\dots]: 4.01 < m(\underline{a}) < \sqrt{20}\}$$ is finite for any parameter $s$ with 
$$g(s)=\frac{|I(a_1,\dots,a_n,3,3,1)|^s + |I(a_1,\dots,a_n,2,1,3)|^s}{|I(a_1,\dots,a_n)|^s}<1$$
and 
$$h(s)=\frac{|I(a_1,\dots,a_n,2,3)|^s+|I(a_1,\dots,a_n,1,1,3)|^s}{|I(a_1,\dots,a_n)|^s}<1$$ 
for all $(a_1,\dots,a_n)\in\bigcup\limits_{k\in\mathbb{N}}\{1,2,3\}^k$. 

We saw in Subsections \ref{ss.sqrt13-3.84}, \ref{ss.3.84-sqrt20} and \ref{ss.3.9-4.01} that 
$$g(s)\leq \left(\frac{1}{255}\right)^s + \frac{|I(a_1,\dots,a_n,2,1,3)|^s}{|I(a_1,\dots,a_n)|^s} \quad \textrm{ and } \quad h(s)\leq (0.016134)^s + \left(\frac{1}{63}\right)^s$$ 
Since the recurrence $q_{j+2}=a_{j+2}q_{j+1}+q_j$ implies that 
\begin{equation*}
\frac{|I(a_1,\dots,a_n,2,1,3)|}{|I(a_1,\dots,a_n)|} = \frac{r+1}{(4r+11)(5r+14)}
\end{equation*}
where $r=q_{n-1}/q_n\in(0,1)$, and since $\frac{r+1}{(4r+11)(5r+14)}\leq 0.007043$ for all $0\leq r\leq 1$, we get 
$$g(s)\leq \left(\frac{1}{255}\right)^s + (0.007043)^s$$ 

Thus, $\max\{g(0.167655), h(0.167655)\} < 0.9999$ and, \emph{a fortiori}, the $(0.167655)$-Hausdorff measure of 
$$K=\{[a_0;a_1,\dots]: 4.01 < m(\underline{a}) < \sqrt{20}\}$$ 
is finite. In particular, we proved the following result: 
\begin{proposition}\label{p.gap-Cantor-sqrt13-II}$(M\setminus L)\cap(4.01, \sqrt{20})\subset K(\{1,2,3\})+K$ where $K$ is a set of Hausdorff dimension $HD(K)<0.167655$.
\end{proposition} 

As usual, this proposition and Hensley's estimate \cite{He} for $HD(K(\{1,2,3\}))$ yields:  
\begin{equation}\label{e.prova-v}
HD((M\setminus L)\cap(4.01,\sqrt{20}))  
< 0.705661 + 0.167655 = 0.873316
\end{equation}

\subsection{Heuristic estimates for $HD((M\setminus L)\cap(\sqrt{20}, \sqrt{21}))$}

Our Proposition \ref{p.gap-Cantor-sqrt20} above implies that 
$$HD((M\setminus L)\cap(\sqrt{20}, \sqrt{21})) < HD(X_4(\{14, 41, 24, 42\}))+0.172825$$
where $X_4(\{14,41,24,42\}):=\{[0;\gamma]:\gamma\in \{1,2,3,4\}^{\mathbb{N}} \textrm{ doesn't contain } 14,  41, 24, 42\}$. 

After running Jenkinson--Pollicott algorithm, one seems to get that 
$$HD(X_4(\{14, 41, 24, 42\}))<0.715$$
and, \emph{a fortiori},  
\begin{equation}\label{e.prova-vi}
HD((M\setminus L)\cap(\sqrt{20}, \sqrt{21})) < 0.888
\end{equation}

\subsection{Global empirical estimate for $HD(M\setminus L)$}

By \eqref{e.prova-i}, \eqref{e.prova-ii}, \eqref{e.prova-iii}, \eqref{e.prova-iv}, \eqref{e.prova-v} and \eqref{e.prova-vi}, we have that 
$$HD((M\setminus L)\cap (-\infty, \sqrt{21})) < 0.888$$

On the other hand, Freiman \cite{Fr73b} and Schecker \cite{Sch77} proved that $[\sqrt{21},\infty)\subset L$. Therefore, $(M\setminus L)\cap [\sqrt{21},\infty)=\emptyset$. 

It follows $HD(M\setminus L)=HD((M\setminus L)\cap (-\infty, \sqrt{21}))<0.888$, the empirical bound announced in \eqref{e.prova}.

\end{document}